\documentclass{elsarticle}
\usepackage{amsmath}
\usepackage{amsfonts}
\usepackage{amsthm}
\usepackage{mathtools}
\usepackage{esint}

\DeclarePairedDelimiter{\floor}{\lfloor}{\rfloor}

\numberwithin{equation}{section}

\newtheorem{theorem}{Theorem}[section]
\newtheorem{proposition}[theorem]{Proposition}
\newtheorem{corollary}[theorem]{Corollary}
\newtheorem{definition}[theorem]{Definition}
\newtheorem{lemma}[theorem]{Lemma}
\newtheorem{remark}[theorem]{Remark}

\newtheorem{condition}[theorem]{Condition}

\begin{document}
\begin{frontmatter}

\title{Fubini theorem in noncommutative geometry\tnoteref{arc}}

\author{F.~Sukochev}
\ead{f.sukochev@unsw.edu.au}
\address{School of Mathematics and Statistics, University of New South Wales,\\ Kensington, Australia}
\author{D.~Zanin}
\ead{d.zanin@unsw.edu.au}
\address{School of Mathematics and Statistics, University of New South Wales,\\ Kensington, Australia}
\tnotetext[arc]{Research supported by the ARC.}

\begin{abstract} We discuss the Fubini formula in Alain Connes' noncommutative geometry. We present a sufficient condition on spectral triples for which a Fubini formula holds true. The condition is natural and related to heat semigroup asymptotics. We provide examples of spectral triples for which the Fubini formula fails.
\end{abstract}
\date{}
\end{frontmatter}

\section{Introduction}

Fix throughout a separable infinite dimensional Hilbert space $H.$ We let $\mathcal{L}(H)$ denote the algebra of all bounded operators on $H.$ For a compact operator $T$ on $H,$ let $\lambda(k,T)$ and $\mu(k,T)$ denote its $k-$th eigenvalue\footnote{The eigenvalues are counted with 
algebraic multiplicities and arranged so that their absolute values are non-increasing.} and $k-$th singular value (these are the eigenvalues of $|T|$). 

We let $\mathcal{L}_{1,\infty}$ denote the principal ideal in $\mathcal{L}(H)$ generated by the operator ${\rm diag}(\{\frac1{k+1}\}_{k\geq0}).$ Equivalently,
$$\mathcal{L}_{1,\infty}=\{T\in\mathcal{L}(H): \mu(k,T)=O(\frac1{k+1})\}.$$
Note that our notation differs from the one used in \cite{Connes}.

\medskip

The following result is proposed on p.~563 in \cite{Connes}.

\begin{proposition}\label{cfub} Let $(1+D_1^2)^{-p_1/2},(1+D_2^2)^{-p_2/2}\in\mathcal{L}_{1,\infty}$ and let $T_1,T_2\in\mathcal{L}(H).$ If one of the elements $T_1(1+D_1^2)^{-p_1/2},$ $T_2(1+D_2^2)^{-p_2/2}$ is convergent, then
\begin{multline}\label{cfub formula}
\Gamma(1+\frac{p_1+p_2}{2}){\rm Tr}_{\omega}((T_1\otimes T_2)(1+D_1^2\otimes 1+1\otimes D_2^2)^{-(p_1+p_2)/2})=\\
=\Gamma(1+\frac{p_1}{2}){\rm Tr}_{\omega}(T_1(1+D_1^2)^{-p_1/2})\cdot \Gamma(1+\frac{p_2}{2}){\rm Tr}_{\omega}(T_2(1+D_2^2)^{-p_2/2})
\end{multline}
holds for some (Dixmier) trace ${\rm Tr}_{\omega}$ on $\mathcal{L}_{1,\infty}.$
\end{proposition}

The wording on p.~563 in \cite{Connes} is that \lq\lq one of the two terms is convergent\rq\rq is open for interpretation. One possible interpretation is that the operator $T_1(1+D_1^2)^{-p_1/2}$ (or $T_2(1+D_2^2)^{-p_2/2}$) is Tauberian. Recall that an operator $A\in\mathcal{L}_{1,\infty}$ is called Tauberian if there exists a limit
$$\lim_{n\to\infty}\frac1{\log(n+2)}\sum_{k=0}^n\lambda(k,A) = c$$
or, in a form convenient for comparison further below,
$$
\sum_{k=0}^n\lambda(k,A) = c \cdot \log(n+1) + o(\log(n+1)) .
$$
We have therefore rephrased the proposition as one of the two operators is Tauberian. 

The functional
$$
T \mapsto \Gamma(1+\frac{p}{2}) {\rm Tr}_{\omega}(T(1+D^2)^{-p/2})
$$
is considered the $p$-dimensional integral in Connes' noncommutative geometry \cite{Connes}. That $T(1+D^2)^{-p/2}$ is Tauberian implies that the functional is independent of which Dixmier trace ${\rm Tr}_{\omega}$ is used to define the functional, a property called measurability. The result proposed on p. 563 in \cite{Connes} is a Fubini formulation for noncommutative geometry, emulating the classical Fubini theorem where the integral on the product space is calculated from the product of the integrals provided one of the integral exists.

In recent personal communication, Professor Connes has kindly explained to the authors that the convergence he had in mind is \lq\lq the convergence in the theta function formula (which in \cite{Connes} is 4 lines above 2.Example a)). The assumed theta convergence is clearly stronger than the convergence of its Cesaro means and all the counter examples of the paper are about this nuance. As shown in Lemma \ref{connes lemma} and Theorem \ref{dixmier fubini}, the Fubini formula indeed holds under theta convergence, by paying attention to the choice of the limiting processes, this is due to Professor Connes and we are grateful for his permission to include his proof into the paper.\rq\rq

Our aim is to study the Fubini formula in detail. We show that the proposal in Proposition~\ref{cfub} does not hold under the condition that one of the terms is Tauberian. It does not hold either with an amended condition that both terms are Tauberian. It does not hold if we ask if one or both of the terms satisfy the stronger condition that
$$
\sum_{k=0}^n\lambda(k,T(1+D^2)^{-p/2}) = c \cdot \log(n+1) + O(1) .
$$
However, we show that there are natural conditions on the terms such that the Fubini formula as stated does hold. As explained above, one of them (see Condition \ref{dixmier regular} below) is also due to Professor Connes.

To state our results we need some definitions. The following terminology was recently introduced in \cite{CRSZ}.

\begin{definition} An operator $A\in\mathcal{L}_{1,\infty}$ is universally measurable if $\varphi(A)$ does not depend on the normalised\footnote{A trace on $\mathcal{L}_{1,\infty}$ is a unitarily invariant linear functional on $\mathcal{L}_{1,\infty}.$ It is normalised if $\varphi({\rm diag}(\{1,\frac12,\frac13,\cdots\}))=1.$} trace $\varphi$ on $\mathcal{L}_{1,\infty}.$ Equivalently (see Theorem~\ref{spectral})
$$
\sum_{k=0}^n\lambda(k,A) = c \cdot \log(n+1) + O(1) .
$$
\end{definition}

Clearly being universally measurable is stronger than being Tauberian. Proposition~\ref{cfub} is false if we show that the same proposition is false for universally measurable operators. 

\begin{definition} We say that a $(p,\infty)-$summable spectral triple\footnote{We refer to \cite{Connes} for the definition of a spectral triple.} $(\mathcal{A},H,D)$ admits a noncommutative integral if, for every $T\in\mathcal{A},$ the operator $T(1+D^2)^{-p/2}$ is universally measurable. In this case, we set
$$\fint(T)=\varphi(T(1+D^2)^{-p/2}),\quad T\in\mathcal{A}$$
for every normalised trace $\varphi$ on $\mathcal{L}_{1,\infty}.$
\end{definition}

The following condition is an analogue of the heat semigroup asymptotics found in Lemma 1.9.2 in \cite{Gilkey}. It is satisfied by all commutative spectral triples of Riemannian manifolds (see the proof of Proposition 3.23 and Theorem 3.24 in \cite{Rosenberg}). Noncommutative tori also satisfies this condition (see the proof of Corollary \ref{torus cor}).

\begin{condition}\label{former regular} $(\mathcal{A},H,D)$ is a $(p,\infty)-$summable spectral triple such that, for every $T\in\mathcal{A},$ there exists $\varepsilon>0$ such that
\begin{equation}\label{ct def}
{\rm Tr}(Te^{-(tD)^2})=\frac{c(T)}{t^p}+O(\frac1{t^{p-\varepsilon}}),\quad t\to0.
\end{equation}
\end{condition}

Our main Fubini theorem can be stated as follows.

\begin{theorem}\label{fubini} Suppose the spectral triples $(\mathcal{A}_1,H_1,D_1)$ and $(\mathcal{A}_2,H_2,D_2)$ satisfy the Condition \ref{former regular}. Then
\begin{enumerate}[{\rm (a)}]
\item $(\mathcal{A}_1,H_1,D_1)$ and $(\mathcal{A}_2,H_2,D_2)$ admit a noncommutative integral.
\item $(\mathcal{A}_1\otimes\mathcal{A}_2,H_1\otimes H_2,(D_1^2\otimes 1+1\otimes D_2^2)^{1/2})$ satisfies the Condition \ref{former regular} and admits a noncommutative integral.
\item For every $T_1\in\mathcal{A}_1,$ $T_2\in\mathcal{A}_2,$ and for every normalised trace $\varphi$ on $\mathcal{L}_{1,\infty},$ we have
$$\Gamma(1+\frac{p_1+p_2}{2})\varphi((T_1\otimes T_2)(1+D_1^2\otimes 1+1\otimes D_2^2)^{-\frac{p_1+p_2}{2}})=$$
$$=\Gamma(1+\frac{p_1}{2})\varphi(T_1(1+D_1^2)^{-\frac{p_1}{2}})\cdot\Gamma(1+\frac{p_2}{2})\varphi(T_2(1+D_2^2)^{-\frac{p_2}{2}}).$$
\end{enumerate}
\end{theorem}

In particular, a Fubini formula holds for noncommutative tori, for sphere $\mathbb{S}^2$ and for the quantum group SU$_q(2).$ The proofs for noncommutative tori extend the idea used in the proof of the main result of \cite{LPS}.

\begin{corollary}\label{torus cor} Let $(\mathcal{A}_1,H_1,D_1)$ and $(\mathcal{A}_2,H_2,D_2)$ be noncommutative tori. For every $T_1\in\mathcal{A}_1,$ $T_2\in\mathcal{A}_2,$ and for every normalised trace $\varphi$ on $\mathcal{L}_{1,\infty},$ we have
$$\Gamma(1+\frac{p_1+p_2}{2})\varphi((T_1\otimes T_2)(1+D_1^2\otimes 1+1\otimes D_2^2)^{-\frac{p_1+p_2}{2}})=$$
$$=\Gamma(1+\frac{p_1}{2})\varphi(T_1(1+D_1^2)^{-\frac{p_1}{2}})\cdot\Gamma(1+\frac{p_2}{2})\varphi(T_2(1+D_2^2)^{-\frac{p_2}{2}}).$$
\end{corollary}

\begin{corollary}\label{sphere suq2 corollary} Let $(\mathcal{A}_1,H_1,D_1)$ and $(\mathcal{A}_2,H_2,D_2)$ be spectral triples which correspond either to sphere $\mathbb{S}^2$ or to the quantum group SU$_q(2).$ For every $T_1\in\mathcal{A}_1,$ $T_2\in\mathcal{A}_2,$ and for every normalised trace $\varphi$ on $\mathcal{L}_{1,\infty},$ we have
$$\Gamma(1+\frac{p_1+p_2}{2})\varphi((T_1\otimes T_2)(1+D_1^2\otimes 1+1\otimes D_2^2)^{-\frac{p_1+p_2}{2}})=$$
$$=\Gamma(1+\frac{p_1}{2})\varphi(T_1(1+D_1^2)^{-\frac{p_1}{2}})\cdot\Gamma(1+\frac{p_2}{2})\varphi(T_2(1+D_2^2)^{-\frac{p_2}{2}}).$$
\end{corollary}

Condition \ref{dixmier regular}, Theorems \ref{connes lemma} and \ref{dixmier fubini} below were suggested by Professor Connes. We are grateful for this valuable addition to the paper.

\begin{condition}\label{dixmier regular} $(\mathcal{A},H,D)$ is a $(p,\infty)-$summable spectral triple such that, for every $T\in\mathcal{A},$ we have
\begin{equation}\label{ctw def}
t^p{\rm Tr}(Te^{-(tD)^2})\to c(T),\quad t\to0.
\end{equation}
\end{condition}

In what follows, we use a notation $\omega^u=\omega\circ P_u,$ $u>0.$ We refer the reader to Section \ref{prel section} for the definition of Dixmier traces.

\begin{theorem}\label{connes lemma} Let $\omega=\omega\circ M$ be a state on $L_{\infty}(0,\infty).$ Suppose that the spectral triple $(\mathcal{A}_1,H_1,D_1)$ (or $(\mathcal{A}_2,H_2,D_2)$) satisfies Condition \ref{dixmier regular}. For every $T_1\in\mathcal{A}_1,$ $T_2\in\mathcal{A}_2,$ we have
$$\Gamma(1+\frac{p_1+p_2}{2}){\rm Tr}_{\omega^{p_1+p_2}}((T_1\otimes T_2)(1+D_1^2\otimes 1+1\otimes D_2^2)^{-\frac{p_1+p_2}{2}})=$$
$$=\Gamma(1+\frac{p_1}{2}){\rm Tr}_{\omega^{p_1}}(T_1(1+D_1^2)^{-\frac{p_1}{2}})\cdot\Gamma(1+\frac{p_2}{2}){\rm Tr}_{\omega^{p_2}}(T_2(1+D_2^2)^{-\frac{p_2}{2}}).$$
\end{theorem}

Theorem \ref{connes lemma} allows us to state another version of Fubini theorem as follows.

\begin{theorem}\label{dixmier fubini} Suppose that the spectral triple $(\mathcal{A}_1,H_1,D_1)$ (or $(\mathcal{A}_2,H_2,D_2)$) satisfies the Condition \ref{dixmier regular}. For every $T_1\in\mathcal{A}_1,$ $T_2\in\mathcal{A}_2,$ and for every Dixmier trace ${\rm Tr}_{\omega}\in\mathfrak{M}$\footnote{Here, $\mathfrak{M}$ is a subclass of Dixmier traces specified in the next section.} on $\mathcal{L}_{1,\infty},$ we have
$$\Gamma(1+\frac{p_1+p_2}{2}){\rm Tr}_{\omega}((T_1\otimes T_2)(1+D_1^2\otimes 1+1\otimes D_2^2)^{-\frac{p_1+p_2}{2}})=$$
$$=\Gamma(1+\frac{p_1}{2}){\rm Tr}_{\omega}(T_1(1+D_1^2)^{-\frac{p_1}{2}})\cdot\Gamma(1+\frac{p_2}{2}){\rm Tr}_{\omega}(T_2(1+D_2^2)^{-\frac{p_2}{2}}).$$
\end{theorem}

It is important to note the difference between Conditions \ref{former regular} and \ref{dixmier regular} and the difference between the assertions of Theorems \ref{fubini} and \ref{dixmier fubini}. Indeed, Theorem \ref{fubini} holds for arbitrary traces on $\mathcal{L}_{1,\infty},$ while Theorem \ref{dixmier fubini} holds for a certain subclass $\mathfrak{M}$ in the class of Dixmier traces. Theorem \ref{dixmier fubini} does not hold for some Dixmier traces outside of the subclass $\mathfrak{M}.$

Condition~\ref{former regular} is stronger than universal measurability. Our second result complements Theorem \ref{fubini} by stating that universal measurability is not sufficient for a Fubini theorem. In fact, the counterexample involves the nicest possible situation where the noncommutative integral is a normal functional on the algebra $\mathcal{A}.$

\begin{definition} Suppose that $(\mathcal{A},H,D)$ admits a noncommutative integral. We say that the noncommutative integral is normal if the mapping
$$T\to\fint(T),\quad T\in\mathcal{A},$$
is continuous in the weak operator topology.
\end{definition}

\begin{theorem}\label{hard example} There exists a $(1,\infty)-$summable spectral triple $(\mathcal{A},H,D)$ such that
\begin{enumerate}[{\rm (a)}]
\item\label{harda} $D$ has simple spectrum $\mathbb{Z}_+.$
\item\label{hardb} $\mathcal{A}$ is generated by a unitary operator $U$ and is finite dimensional.
\item\label{hardc} $(\mathcal{A},H,D)$ admits a normal noncommutative integral.
\item\label{hardd} $(\mathcal{A}\otimes\mathcal{A},H\otimes H,(D^2\otimes 1+1\otimes D^2)^{1/2})$ admits a normal noncommutative integral.
\item\label{harde} For every normalised trace $\varphi$ on $\mathcal{L}_{1,\infty},$ we have
$$\varphi((U\otimes U^{-1})(1+D^2\otimes 1+1\otimes D^2)^{-1})\neq0,\quad \varphi(U(1+D^2)^{-\frac{1}{2}})=\varphi(U^{-1}(1+D^2)^{-\frac{1}{2}})=0.$$
\end{enumerate}
\end{theorem}

\begin{corollary}\label{ex for pos} In the setting of Theorem \ref{hard example}, there exists a positive element $T\in\mathcal{A}$ such that
$$\varphi((T\otimes T)(1+D^2\otimes 1+1\otimes D^2)^{-1})\neq\frac{\pi}{4}(\varphi(T(1+D^2)^{-\frac{1}{2}}))^2$$
for every normalised trace $\varphi$ on $\mathcal{L}_{1,\infty}.$
\end{corollary}

Our second counterexample shows that volume in noncommutative geometry is not necessarily well behaved under the product operation on spectral triples, even with the strong condition of universal measurability. 

\begin{theorem}\label{second example} There exists an operator $D$ such that
\begin{enumerate}[{\rm (a)}]
\item $(1+D^2)^{-1/2}\in\mathcal{L}_{1,\infty}$ is universally measurable.
\item $(1+D^2\otimes 1+D^2\otimes 1)^{-1}\in\mathcal{L}_{1,\infty}$ is universally measurable.
\item For every normalised trace $\varphi$ on $\mathcal{L}_{1,\infty},$ we have
$$\varphi((1+D^2\otimes 1+1\otimes D^2)^{-1})>\frac{\pi}{4}(\varphi((1+D^2)^{-\frac{1}{2}}))^2.$$
\end{enumerate}
\end{theorem}

Our final counterexample (proved in Appendix \ref{easy}) shows that it does not suffice to impose Condition \ref{former regular} only on one spectral triple. It also shows that the assertion of Theorem \ref{dixmier fubini} fails for some Dixmier trace (outside of the class $\mathfrak{M}$).

\begin{theorem} There exist spectral triples $(\mathcal{A}_1,l_2,D)$ and $(\mathbb{C},l_2,D)$ such that
\begin{enumerate}[{\rm (a)}]
\item $D$ has simple spectrum $\mathbb{Z}_+.$ In particular, the spectral triple $(\mathbb{C},l_2,D)$ satisfies the Condition \ref{former regular}.
\item There exists an operator $T_1\in\mathcal{A}_1$ and a (Dixmier) trace $\varphi$ on $\mathcal{L}_{1,\infty}$ such that
$$\varphi((T_1\otimes 1)(1+D^2\otimes 1+1\otimes D^2)^{-1})\neq\frac{\pi}{4}\varphi(T_1(1+D^2)^{-1/2}).$$
\end{enumerate}
\end{theorem}

We are grateful to Professor Connes for his kind explanation of noncommutative Fubini formula. We are also grateful to A. Carey, Y. Kuznetsova, S. Lord and A. Rennie for useful comments on the earlier versions of this manuscript. The authors would also like to mention that the idea of the proof of Lemma \ref{kalton} was conveyed to them by the late N. Kalton.

\section{Preliminaries}\label{prel section}

The standard trace on $\mathcal{L}(H)$ is denoted by ${\rm Tr}.$  Fix an orthonormal basis in $H$ (the particular choice of basis is inessential). We identify the  algebra $l_{\infty}$ of bounded sequences with the subalgebra of all diagonal operators with respect to the chosen basis. We set $l_{1,\infty}=\mathcal{L}_{1,\infty}\cap l_{\infty}.$ For a given sequence $x\in l_{\infty},$ we denote the corresponding diagonal operator by ${\rm diag}(x).$

\begin{definition}\label{trace def} A trace on $\mathcal{L}_{1,\infty}$ is a unitarily invariant linear functional $\varphi:\mathcal{L}_{1,\infty}\to\mathbb{C}$.
\end{definition}

Traces on $\mathcal{L}_{1,\infty}$ satisfying the condition
$$\varphi(TS)=\varphi(ST),\quad T\in\mathcal{L}_{1,\infty}, S\in\mathcal{L}(H).$$
The latter may be reinterpreted as the vanishing of the linear functional $\varphi$ on the commutator subspace
$$[\mathcal{L}_{1,\infty},\mathcal{L}(H)]={\rm span}\{ST-TS,\quad T\in\mathcal{L}_{1,\infty}, S\in\mathcal{L}(H)\}.$$

An example of a trace on $\mathcal{L}_{1,\infty}$ is a Dixmier trace that we now explain (we use the definition from \cite{SS}, which, according to Theorem 17 in \cite{SS}, produces exactly the same class of traces on $\mathcal{L}_{1,\infty}$ as the one in \cite{Connes}). Namely, for every ultrafilter $\omega,$ the functional ${\rm Tr}_{\omega}$ defined on the positive cone of $\mathcal{L}_{1,\infty}$ by the formula
\begin{equation}\label{dixtrdef}
{\rm Tr}_{\omega}(A)=\lim_{n\to\omega}\frac1{\log(n+2)}\sum_{k=0}^n\mu(k,A),\quad 0\leq A\in\mathcal{L}_{1,\infty},
\end{equation}
is additive and, therefore, extends to a positive unitarily invariant linear functional on $\mathcal{L}_{1,\infty}$ called a Dixmier trace.

In order to properly state Theorem \ref{dixmier fubini}, we need a smaller subclass $\mathfrak{M}$ of Dixmier traces. Let $\omega$ be a state on the algebra $L_{\infty}(0,\infty)$ which satisfies the condition $\omega=\omega\circ M$ (see p.35 in \cite{BF}). Here, the linear operator $M:L_{\infty}(0,\infty)\to L_{\infty}(0,\infty)$ is given by the formula
$$(Mx)(t)=\frac1{\log(t)}\int_1^tx(s)\frac{ds}{s},\quad t>0.$$
The functional ${\rm Tr}_{\omega}$ is defined on the positive cone of $\mathcal{L}_{1,\infty}$ by the formula
$${\rm Tr}_{\omega}(A)=\omega\Big(t\to \frac1{\log(1+t)}\int_0^t\mu(s,A)ds\Big),\quad 0\leq A\in\mathcal{L}_{1,\infty}.$$
This functional is additive and, therefore, extends to a positive unitarily invariant linear functional on $\mathcal{L}_{1,\infty}$ (see e.g. \cite{BF}).

Let the group $(\mathbb{R}_+,\cdot)$ act on $L_{\infty}(0,\infty)$ by the formula $u\to P_u,$ $(P_ux)(t)=x(t^u),$ $u,t>0.$ Note that $M\circ P_u=P_u\circ M$ (see a similar formula (3) in \cite{SUZ1}). In particular, ${\rm Tr}_{\omega^u}={\rm Tr}_{\omega\circ P_u}$ is also a positive unitarily invariant linear functional on $\mathcal{L}_{1,\infty}.$ We set
$$\mathfrak{M}=\{{\rm Tr}_{\omega}:\ \omega=\omega\circ M,\quad \omega=\omega\circ P_u,\ u>0\}.$$
It is important to note that $\omega$ in this paragraph can never be an ultrafilter. However, ${\rm Tr}_{\omega}$ is still a Dixmier trace according to the main result of \cite{SS}.

The following assertion is Theorem 3 in \cite{BF}.

\begin{theorem}\label{bf main theorem} Let $\omega=\omega\circ M$ be a state on the algebra $L_{\infty}(0,\infty).$ If the triple $(\mathcal{A},H,D)$ is $(p,\infty)-$summable, then
$$\Gamma(1+\frac{p}{2}){\rm Tr}_{\omega\circ P_p}(T(1+D^2)^{-\frac{p}{2}})=\omega\Big(t\to t^{-p}{\rm Tr}(Te^{-t^{-2}D^2})\Big),\quad T\in\mathcal{A}.$$
\end{theorem}

The following theorem provides the convenient spectral description for universally measurable operators referred to earlier. It was originally proved in \cite{DFWW} for normal operators and, then in \cite{Kcomm} and \cite{DK} for arbitrary operators (see also \cite{KLPS}). For accessible proof, we refer the interested reader to Theorem 10.1.3 in \cite{LSZ} and its proof in Chapter 5 in \cite{LSZ}.

\begin{theorem}\label{spectral} For $A\in\mathcal{L}_{1,\infty},$ the following conditions are equivalent.
\begin{enumerate}[{\rm (a)}]
\item We have
$\varphi(A)=c$
for every normalised trace $\varphi$ on $\mathcal{L}_{1,\infty}.$
\item We have
$$\sum_{m=0}^n\lambda(m,A)=c\cdot\log(n+1)+O(1),\quad n\geq0.$$
\end{enumerate}
In particular, $A\in[\mathcal{L}_{1,\infty},\mathcal{L}(H)]$ if and only if
$$\sum_{m=0}^n\lambda(m,A)=O(1),\quad n\geq0.$$
\end{theorem}

Every universally measurable operator is Tauberian (that is, Dixmier-measurable \cite{SS}). For various sorts of measurability results in noncommutative geometry, we refer the interested reader to papers \cite{SUZ1,SUZ2} and to the book \cite{LSZ}. 

\section{Proof of Theorems \ref{fubini} and \ref{dixmier fubini}}

\begin{lemma}\label{crsz-like1} Let $\Phi:(0,\infty)\to(0,1)$ be such that
\begin{enumerate}[{\rm (i)}]
\item\label{cri} $\Phi$ is convex, decreasing and positive.
\item\label{crii} $\Phi(0)=1$ and
$$\int_1^{\infty}\Phi(t)\frac{dt}{t}<\infty,\quad\int_0^1(\frac1t-1)(1-\Phi(t))dt<\infty.$$
\end{enumerate}
For every $0\leq V\in\mathcal{L}_{1,\infty},$ we have
$$\|\min\{V\Phi((nV)^{-1}),\frac1n\}\|_1=O(1),\quad n\to\infty$$
$$\|(V-\frac1n)_+(1-\Phi((nV)^{-1}))\|_1=O(1),\quad n\to\infty.$$
\end{lemma}
\begin{proof} It is easy using \eqref{cri} to check that the functions
$$x\to x\Phi(x^{-1}),\quad x\to(x-1)_+(1-\Phi(x^{-1}))$$
increase on $(0,\infty).$ For simplicity of computations, let $\|V\|_{1,\infty}=1.$ Let $W\in\mathcal{L}_{1,\infty}$ be an operator commuting with $V$ such that $0\leq V\leq W$ and such that $\mu(k,W)=\frac1{k+1},$ $k\geq0.$ It follows that
$$\|\min\{V\Phi((nV)^{-1}),\frac1n\}\|_1\leq\|\min\{W\Phi((nW)^{-1}),\frac1n\}\|_1\leq$$
$$\leq\|\{\frac1{k+1}\Phi(\frac{k+1}{n})\}_{k=n}^{\infty}\|_1+\|\{\frac1n\}_{k=0}^{n-1}\|_1\leq\int_n^{\infty}\Phi(\frac{t}{n})\frac{dt}{t}+1\stackrel{\eqref{crii}}{=}O(1)$$
and, similarly,
$$\|(V-\frac1n)_+(1-\Phi((nV)^{-1}))\|_1\leq\|(W-\frac1n)_+(1-\Phi((nW)^{-1}))\|_1\leq$$
$$\leq\|\{(\frac1{k+1}-\frac1n)(1-\Phi(\frac{k+1}{n}))\}_{k=0}^{n-1}\|_1\leq\int_0^n(\frac1t-\frac1n)(1-\Phi(\frac{t}{n}))dt\stackrel{\eqref{crii}}{=}O(1).$$
\end{proof}

The following lemma extends Proposition 6 in \cite{CRSZ}.

\begin{lemma}\label{crsz-like2} Let $0\leq V\in\mathcal{L}_{1,\infty}$ and let $A\in\mathcal{L}(H).$ Let $\Phi$ be as in Lemma \ref{crsz-like1}. The following conditions are equivalent
\begin{enumerate}[{\rm (a)}]
\item $\varphi(AV)=c$ for every normalised trace $\varphi$ on $\mathcal{L}_{1,\infty}.$
\item We have
$${\rm Tr}(AV\Phi((nV)^{-1}))=c\log(n)+O(1),\quad n\to\infty.$$
\end{enumerate}
\end{lemma}
\begin{proof} It is clear that
$$AV\Phi((nV)^{-1})-AVE_V[\frac1n,\infty)=$$
$$=AV\Phi((nV)^{-1})E_V[0,\frac1n)+AV(\Phi((nV)^{-1})-1)E_V[\frac1n,\infty).$$
Therefore,
$$|{\rm Tr}(AV\Phi((nV)^{-1}))-{\rm Tr}(AVE_V[\frac1n,\infty))|\leq\|AV\Phi((nV)^{-1})-AVE_V[\frac1n,\infty)\|_1\leq$$
$$\leq\|A\|_{\infty}\Big(\|V\Phi((nV)^{-1})E_V[0,\frac1n)\|_1+\|V(\Phi((nV)^{-1})-1)E_V[\frac1n,\infty)\|_1\Big)\leq$$
$$\leq\|A\|_{\infty}\Big(\|V\Phi((nV)^{-1})E_V[0,\frac1n)\|_1+\|(V-\frac1n)(\Phi((nV)^{-1})-1)E_V[\frac1n,\infty)\|_1+$$
$$+\frac1n\|(\Phi((nV)^{-1})-1)E_V[\frac1n,\infty)\|_1\Big)\leq$$
$$\leq\|A\|_{\infty}\Big(\|\min\{V\Phi((nV)^{-1}),\frac1n\}\|_1+\|(V-\frac1n)_+(\Phi((nV)^{-1})-1)\|_1+\frac1n\|E_V[\frac1n,\infty)\|_1\Big).$$
It follows from Lemma \ref{crsz-like1} that
$${\rm Tr}(AV\Phi((nV)^{-1}))-{\rm Tr}(AVE_V[\frac1n,\infty))=O(1),\quad n\to\infty.$$
It follows now from Lemma 8 in \cite{CRSZ} that
$${\rm Tr}(AV\Phi((nV)^{-1}))-\sum_{k=0}^n\lambda(k,AV)=O(1).$$
The assertion follows now from Theorem \ref{spectral}.
\end{proof}

\begin{lemma}\label{integral exists} If a spectral triple $(\mathcal{A},H,D)$ satisfies the Condition \ref{former regular}, then it admits a noncommutative integral. More precisely, if $(\mathcal{A},H,D)$ is $(p,\infty)-$summable, then
$$c(T)=\Gamma(1+\frac{p}{2})\varphi(T(1+D^2)^{-\frac{p}{2}}),\quad T\in\mathcal{A},$$
for every normalised trace $\varphi$ on $\mathcal{L}_{1,\infty}.$ Here, $c(T)$ is the number which appears in \eqref{ct def}.
\end{lemma}
\begin{proof} It follows from \eqref{ct def} that
$${\rm Tr}(Te^{-t^2(1+D^2)})=\frac{c(T)}{t^p}+O(\frac1{t^{p-\varepsilon}}),\quad t\to0.$$
Substituting $t^{\frac1p}$ instead of $t,$ we infer that
\begin{equation}\label{ie1}
{\rm Tr}(Te^{-(t(1+D^2)^{\frac{p}{2}})^{\frac{2}{p}}})=\frac{c(T)}{t}+O(\frac1{t^{1-\varepsilon}}),\quad t\to0.
\end{equation}
Set
$$\Phi(s)=\frac1{\Gamma(1+\frac{p}{2})}\int_s^{\infty}e^{-t^{\frac{2}{p}}}dt,\quad s>0.$$
We have
$$\int_s^1e^{-(t(1+D^2)^{\frac{p}{2}})^{\frac{2}{p}}}dt=(1+D^2)^{-p/2}\Big(\Phi(s(1+D^2)^{p/2})-\Phi((1+D^2)^{p/2})\Big) ,$$
where the integral is understood in the Bochner sense in $\mathcal{L}_1.$ In particular, we have
$$\int_s^1{\rm Tr}(Te^{-(t(1+D^2)^{\frac{p}{2}})^{\frac{2}{p}}})dt=$$
$$={\rm Tr}(T(1+D^2)^{-\frac{p}{2}}\Phi(s(1+D^2)^{\frac{p}{2}}))-{\rm Tr}(T(1+D^2)^{-\frac{p}{2}}\Phi((1+D^2)^{\frac{p}{2}})).$$
Integrating both sides in \eqref{ie1} over $t\in[s,1]$ and dividing by $\Gamma(1+\frac{p}{2})$ we infer that
$${\rm Tr}(T(1+D^2)^{-\frac{p}{2}}\Phi(s(1+D^2)^{\frac{p}{2}}))=\frac{c(T)}{\Gamma(1+\frac{p}{2})}\log(\frac1s)+O(1),\quad s\to0.$$
Observe, that $\Phi$ satisfies the conditions of Lemma \ref{crsz-like1}. The assertion follows now from Lemma \ref{crsz-like2} (as applied to $V=(1+D^2)^{-\frac{p}{2}}$ and $c=\frac{c(T)}{\Gamma(1+\frac{p}{2})}$).
\end{proof}

\begin{proof}[Proof of Theorem \ref{fubini}] Recall an abstract equality (which holds for all bounded operators $T_1,T_2$)
$$(T_1\otimes T_2)e^{-t^2(D_1^2\otimes 1+1\otimes D_2^2)})=T_1e^{-t^2D_1^2}\otimes T_2e^{-t^2D_2^2}.$$
Take now $T_1\in\mathcal{A}_1$ and $T_2\in\mathcal{A}_2.$ By Condition \ref{former regular}, we have
$${\rm Tr}(T_1e^{-t^2D_1^2})=\frac{c(T_1)}{t^{p_1}}+O(\frac1{t^{p_1-\varepsilon}}),\quad{\rm Tr}(T_2e^{-t^2D_2^2})=\frac{c(T_2)}{t^{p_2}}+O(\frac1{t^{p_2-\varepsilon}}),\quad t\to0.$$
It follows that
$${\rm Tr}((T_1\otimes T_2)e^{-t^2(D_1^2\otimes 1+1\otimes D_2^2)}))=\frac{c(T_1)c(T_2)}{t^{p_1+p_2}}+O(\frac1{t^{p_1+p_2-\varepsilon}})$$
as $t\to0.$ Thus, the spectral triple $(\mathcal{A}_1\otimes\mathcal{A}_2,H_1\otimes H_2,(D_1^2\otimes 1+1\otimes D_2^2)^{1/2})$ satisfies the Condition \ref{former regular} and
$$c(T_1\otimes T_2)=c(T_1)c(T_2).$$
The assertion follows now from Lemma \ref{integral exists}.
\end{proof}

\begin{proof}[Proof of Theorem \ref{connes lemma}] Recall an abstract equality (which holds for all bounded operators $T_1,T_2$)
$$(T_1\otimes T_2)e^{-t^2(D_1^2\otimes 1+1\otimes D_2^2)})=T_1e^{-t^2D_1^2}\otimes T_2e^{-t^2D_2^2}.$$
Take now $T_1\in\mathcal{A}_1$ and $T_2\in\mathcal{A}_2.$ By Condition \ref{dixmier regular}, we have
$${\rm Tr}(T_1e^{-t^2D_1^2})=\frac{c(T_1)}{t^{p_1}}+o(\frac1{t^{p_1}}),\quad t\to0.$$
Taking the trace and replacing $t$ with $t^{-1},$ we obtain that
$$t^{-(p_1+p_2)}{\rm Tr}((T_1\otimes T_2)e^{-t^{-2}(D_1^2\otimes 1+1\otimes D_2^2)}))=(c(T_1)+o(1))\cdot t^{-p_2}{\rm Tr}(T_2e^{-t^{-2}D_2^2}),\quad t\to\infty.$$
In particular, applying $\omega$ to the both sides of the equality, we arrive at
$$\omega\Big(t\to t^{-(p_1+p_2)}{\rm Tr}((T_1\otimes T_2)e^{-t^{-2}(D_1^2\otimes 1+1\otimes D_2^2)}))\Big)=c(T_1)\omega\Big(t\to t^{-p_2}{\rm Tr}(T_2e^{-t^{-2}D_2^2})\Big).$$
It follows from Theorem \ref{bf main theorem} (applied to both sides of the equality) that
$$\Gamma(1+\frac{p_1+p_2}{2}){\rm Tr}_{\omega^{p_1+p_2}}((T_1\otimes T_2)(1+D_1^2\otimes 1+1\otimes D_2^2)^{-\frac{p_1+p_2}{2}})=$$
$$=c(T_1)\cdot\Gamma(1+\frac{p_2}{2}){\rm Tr}_{\omega^{p_2}}(T_2(1+D_2^2)^{-\frac{p_2}{2}}).$$
Again using Theorem \ref{bf main theorem} (applied to the spectral triple $(\mathcal{A}_1,H_1,D_1)$), we infer that
$$c(T_1)=\Gamma(1+\frac{p_1}{2}){\rm Tr}_{\omega^{p_1}}(T_1(1+D_1^2)^{-\frac{p_1}{2}}).$$
This concludes the proof.
\end{proof}

\begin{proof}[Proof of Theorem \ref{dixmier fubini}] If $\omega=\omega\circ P_u,$ $u>0,$ then
$${\rm Tr}_{\omega\circ P_{p_1}}={\rm Tr}_{\omega\circ P_{p_2}}={\rm Tr}_{\omega\circ P_{p_1+p_2}}={\rm Tr}_{\omega}.$$
The assertion follows now from Lemma \ref{connes lemma}.
\end{proof}

\section{Physically relevant examples}

We supply 3 examples which satisfy the Condition \ref{former regular}. The first example is a sphere --- the simplest possible non-flat manifold. The second example is a noncommutative torus. The third and the most technically involved example is the quantum group SU$_q(2).$ 

The following elementary lemma is needed in all 3 examples. We incorporate the proof for convenience of the reader.

\begin{lemma}\label{pois lemma} We have
\begin{enumerate}[{\rm (a)}]
\item\label{pois sphere} $$\sum_{l\in\mathbb{Z}}|l|e^{-l^2t^2}=t^{-2}+O(t^{-1}),\quad t\to0.$$
\item\label{pois torus} $$\sum_{l\in\mathbb{Z}}e^{-l^2t^2}=\frac{\pi^{\frac12}}{t}+O(1),\quad t\to0.$$
\item\label{pois suq2} $$\sum_{l\in\mathbb{Z}}l^2e^{-l^2t^2}=\frac{\pi^{\frac12}}{2t^3}+O(t^{-2}),\quad t\to0.$$
\end{enumerate}
\end{lemma}
\begin{proof} Though the second and third equalities can be derived from the Poisson summation formula, this method gives nothing good for the first equality. We provide an elementary proof of the first equality. The proofs of the second and third are similar.

The function $s\to se^{-s^2t^2}$ admits its maximum at the point $s=\frac1{t\sqrt{2}}.$ Thus, the function increases on the interval $(0,\frac1{t\sqrt{2}})$ and decreases on the interval $(\frac1{t\sqrt{2}},\infty).$ It follows that
$$\sum_{l=1}^{\lfloor\frac1{t\sqrt{2}}\rfloor-1}\int_{l-1}^lse^{-s^2t^2}ds\leq\sum_{l=0}^{\lfloor\frac1{t\sqrt{2}}\rfloor-1}le^{-l^2t^2}\leq\sum_{l=0}^{\lfloor\frac1{t\sqrt{2}}\rfloor-1}\int_l^{l+1}se^{-s^2t^2}ds.$$
Thus,
$$\Big|\sum_{l=0}^{\lfloor\frac1{t\sqrt{2}}\rfloor-1}le^{-l^2t^2}-\int_0^{\lfloor\frac1{t\sqrt{2}}\rfloor}se^{-s^2t^2}ds\Big|\leq\int_{\lfloor\frac1{t\sqrt{2}}\rfloor-1}^{\lfloor\frac1{t\sqrt{2}}\rfloor}se^{-s^2t^2}ds\leq\sup_{s>0}se^{-s^2t^2}=\frac{e^{-\frac12}}{t\sqrt{2}}.$$
Similarly,
$$\Big|\sum_{l=\lceil\frac1{t\sqrt{2}}\rceil+1}^{\infty}le^{-l^2t^2}-\int_{\lceil\frac1{t\sqrt{2}}\rceil}^{\infty}se^{-s^2t^2}ds\Big|\leq\frac{e^{-\frac12}}{t\sqrt{2}}.$$
It follows that
$$\sum_{l=0}^{\infty}le^{-l^2t^2}=\Big(\sum_{l=0}^{\lfloor\frac1{t\sqrt{2}}\rfloor-1}le^{-l^2t^2}\Big)+\Big(\sum_{l=\lceil\frac1{t\sqrt{2}}\rceil+1}^{\infty}le^{-l^2t^2}\Big)+\Big(\sum_{l=\lfloor\frac1{t\sqrt{2}}\rfloor}^{\lceil\frac1{t\sqrt{2}}\rceil}le^{-l^2t^2}\Big)=$$
$$=\Big(\int_0^{\lfloor\frac1{t\sqrt{2}}\rfloor}se^{-s^2t^2}ds+O(t^{-1})\Big)+\Big(\int_{\lceil\frac1{t\sqrt{2}}\rceil}^{\infty}se^{-s^2t^2}ds+O(t^{-1})\Big)+O(t^{-1})=$$
$$=\int_0^{\infty}se^{-s^2t^2}ds+O(t^{-1})=t^{-2}\int_0^{\infty}se^{-s^2}ds+O(t^{-1})=\frac1{2t^2}+O(t^{-1}).$$
It follows that
$$\sum_{l=-\infty}^0|l|e^{-l^2t^2}=\sum_{l=0}^{\infty}le^{-l^2t^2}=\frac1{2t^2}+O(t^{-1}).$$
Adding the last $2$ formulae, we conclude the proof.
\end{proof}

\subsection{Example: sphere $\mathbb{S}^2$}

We briefly recall the construction of a spectral triple on sphere $\mathbb{S}^2.$ Interested reader is referred to \cite{greenbook} for details.

Let $s=(s_1,s_2,s_3)$ be the point on the sphere $\mathbb{S}^2$ expressed in Cartesian coordinates. Define stereographic coordinates by the formula $(x,y)=(\frac{s_1}{1-s_3},\frac{s_2}{1-s_3}).$ Denote $z=x+iy$ and $q(x,y)=1+x^2+y^2.$ The image of Lebesgue measure on sphere under stereographic projection is $4q^{-2}(x,y)dxdy.$ Define unbounded operators $D_1$ and $D_2$ on (the subspace of all Schwartz functions in) the Hilbert space $L_2(\mathbb{R}^2,4q^{-2}(x,y)dxdy)$ by setting $D_1=\frac1i\frac{\partial}{\partial x},$ $D_2=\frac1i\frac{\partial}{\partial y}.$  Consider now a couple $(A,A^*)$ of formally adjoint unbounded operators defined on (the subspace of all Schwartz functions in) the Hilbert space $L_2(\mathbb{R}^2,4q^{-2}(x,y)dxdy)$ by the formula
$$A=\frac12M_q(D_1-iD_2)+\frac{i}{2}M_{\bar{z}},\quad A^*=\frac12M_q(D_1+iD_2)+\frac{i}{2}M_z.$$

Our Hilbert space is $\mathbb{C}^2\otimes L_2(\mathbb{R}^2,4q^{-2}(x,y)dxdy).$ Our von Neumann algebra is $L_{\infty}(\mathbb{S}^2)$ with a smooth subalgebra $C^{\infty}(\mathbb{S}^2).$ Its representation is given by the formula $\pi(f)=1\otimes M_{f\circ{\rm Stereo}^{-1}},$ where ${\rm Stereo}$ denotes stereographic projection. Our Dirac operator is then defined by the formula (see equation (9.52) in \cite{greenbook})
$$D=e_{12}\otimes A+e_{21}\otimes A^*,$$
where $e_{12},e_{21}\in M_2(\mathbb{C})$ are matrix units. It is established in Corollary 9.26 and Proposition 9.28 in \cite{greenbook} that $D$ admits an orthonormal eigenbasis. In particular, $D$ is self-adjoint.

By Corollary 9.29 in \cite{greenbook}, the constructed spectral triple
$$(\pi(L_{\infty}(\mathbb{S}^2)),\mathbb{C}^2\otimes L_2(\mathbb{R}^2,4q^{-2}(x,y)dxdy),D)$$
is $(2,\infty)-$summable. In the following lemma, we show that it satisfies the Condition \ref{former regular}.

\begin{lemma}\label{sphere ok} For every $f\in L_{\infty}(\mathbb{S}^2),$ we have
$${\rm Tr}(\pi(f)e^{-t^2D^2})=t^{-2}\cdot\frac1{4\pi}\int_{\mathbb{S}^2}f(s)ds+O(t^{-1}),\quad t\to0.$$
\end{lemma}
\begin{proof} Recall how the group SU$(2)$ acts on extended complex plane.
\begin{equation}\label{su2 action}
g=
\begin{pmatrix}
a&b\\
-\bar{b}&\bar{a}
\end{pmatrix}\in SU(2),\quad
g(z)=\frac{az+b}{-\bar{b}z+\bar{a}},\quad z\in\mathbb{C}.
\end{equation}

This action results in the unitary representation $\tau$ of the group SU$(2)$ on the Hilbert space $\mathbb{C}^2\otimes L_2(\mathbb{S}^2)$ by the formula
$$\Big(\tau(g)
\begin{pmatrix}
\psi_1\\
\psi_2
\end{pmatrix}\Big)(z)=
\Big(\frac{b\bar{z}+\bar{a}}{\bar{b}z+a}\Big)^{\frac12}\begin{pmatrix}
\psi_1(g^{-1}(z))\\
\psi_2(g^{-1}(z))
\end{pmatrix},
\quad
\psi_1,\psi_2\in L_2(\mathbb{R}^2,4q^{-2}(x,y)dxdy).
$$
The key fact (Proposition 9.27 in \cite{greenbook}) is that Dirac operator $D$ commutes with the $\tau(g)$ for every $g\in SU(2).$

Let $f\in L_{\infty}(\mathbb{S}^2)$ and denote for brevity $F=f\circ{\rm Stereo}^{-1}.$ Let $P$ be a spectral projection of $D.$ Since $P$ commutes with $\tau(g),$ it follows that
$${\rm Tr}((1\otimes M_F)P)={\rm Tr}((1\otimes M_F)\tau(g)\cdot \tau(g^{-1})P)={\rm Tr}((1\otimes M_F)\tau(g)\cdot P\tau(g^{-1}))=$$
$$={\rm Tr}(\tau(g^{-1})(1\otimes M_F)\tau(g)\cdot P).$$
Note that
$$\tau(g^{-1})(1\otimes M_F)\tau(g)=M_{F\circ g}.$$
Thus,
$${\rm Tr}(\pi(f)P)={\rm Tr}((1\otimes M_F)P)={\rm Tr}((1\otimes M_{f\circ g})P),\quad g\in SU(2).$$
Since the latter equality holds for every $g\in SU(2),$ it follows that
$${\rm Tr}(\pi(f)P)=\int_{SU(2)}{\rm Tr}((1\otimes M_{F\circ g})P)dg={\rm Tr}((1\otimes \int_{SU(2)}M_{F\circ g}dg)P),$$
where $dg$ is the Haar measure on SU$(2).$ The action of SU$(2)$ as given in \eqref{su2 action} is conjugated (by means of stereographic projection, see Section 1.4 in \cite{GMS}) to the action of SU$(2)$ on sphere $\mathbb{S}^2$ by rotations. It follows that
$$\int_{SU(2)}(F\circ g)dg=\frac1{4\pi}\int_{\mathbb{S}^2}f(s)ds$$
and, therefore,
$${\rm Tr}(\pi(f)P)=\frac1{4\pi}\int_{\mathbb{S}^2}f(s)ds\cdot {\rm Tr}(P).$$

According to Corollary 9.29 in \cite{greenbook}, spectrum of $D$ is $\mathbb{Z}\backslash\{0\}$ and, for every $l\in\mathbb{Z}\backslash\{0\},$ ${\rm Tr}(E_D\{l\})=|l|.$ It follows that
$${\rm Tr}(\pi(f)e^{-t^2D^2})=\sum_{l\in\mathbb{Z}\backslash\{0\}}{\rm Tr}(\pi(f)e^{-t^2D^2}E_D\{l\})=$$
$$=\sum_{l\in\mathbb{Z}\backslash\{0\}}e^{-t^2l^2}{\rm Tr}(\pi(f)E_D\{l\})=\frac1{4\pi}\int_{\mathbb{S}^2}f(s)ds\cdot\sum_{l\in\mathbb{Z}}|l|e^{-t^2l^2}.$$
The assertion follows now from Lemma \ref{pois lemma} \eqref{pois sphere}.
\end{proof}

\subsection{Example: noncommutative torus}

We briefly recall a spectral triple for the noncommutative torus (originally introduced in Section II.2.$\beta$ in \cite{Connes}). After that, we show that the triple satisfies the Condition \ref{former regular}.

Let $\Theta\in M_p(\mathbb{R}),$ $1<p\in\mathbb{N},$ be an anti-symmetric matrix. Let $A_{\Theta}$ be the universal $*-$algebra generated by unitaries $\{U_k\}_{k=1}^p$ satisfying the conditions
$$U_{k_2}U_{k_1}=e^{i\Theta_{k_1,k_2}}U_{k_1}U_{k_2},\quad 1\leq k_1,k_2\leq p.$$
Define a linear functional $\tau:A_{\Theta}\to\mathbb{C}$ by setting
$$\tau(U_1^{n_1}\cdots U_p^{n_p})=0\mbox{ unless }(n_1,\cdots,n_p)=0.$$
It can be demonstrated that $\tau$ is positive, that is $\tau(x^*x)\geq0$ for $x\in A_{\Theta}.$ We equip linear space $A_{\Theta}$ with an inner product defined by the formula
$$\langle x,y\rangle=\tau(xy^*),\quad x,y\in A_{\Theta}.$$
Natural action $\lambda$ of $A_{\Theta}$ on pre-Hilbert space $(A_{\Theta},\langle\cdot,\cdot\rangle)$ by left multiplications extends to the action on the completed Hilbert space. The weak$^*$ closure of $\lambda(A_{\Theta})$ is denoted by $L_{\infty}(\mathbb{T}^p_{\Theta})$ and $\tau$ extends to a faithful normal tracial state on $L_{\infty}(\mathbb{T}^p_{\Theta}).$ The Hilbert space where $L_{\infty}(\mathbb{T}^p_{\Theta})$ is naturally identified with $L_2(\mathbb{T}^p_{\Theta},\tau).$

A natural spectral triple for the noncommutative torus is given as follows.\footnote{The $C^*-$algebra $\overline{\lambda(A_{\Theta})}^{\|\cdot\|_{\infty}}$ is isomorphic to a universal $C^*-$algebra constructed by Davidson (see pp.166-170 in \cite{Davidson}).} Set $\mathcal{A}=L_{\infty}(\mathbb{T}^p_{\Theta})$ and take $\lambda(A_{\Theta})$ to be the subalgebra of smooth elements. Let $m(p)=2^{\lfloor\frac{p}{2}\rfloor}.$ Set $H=\mathbb{C}^{m(p)}\otimes L_2(\mathbb{T}^p_{\Theta})$ and $\pi(x)=1\otimes M_x,$ $x\in L_{\infty}(\mathbb{T}^p_{\Theta}).$ Define self-adjoint operators $D_k,$ $1\leq k\leq p,$ on the Hilbert space $L_2(\mathbb{T}^p_{\Theta})$ by setting
$$D_k:U_1^{n_1}\cdots U_p^{n_p}\to n_kU_1^{n_1}\cdots U_p^{n_p},\quad (n_1,\cdots,n_p)\in\mathbb{Z}^p.$$
Those operators commute. Dirac operator $D$ acts on the Hilbert space $H$ by the setting
$$D=\sum_{k=1}^p\gamma_k\otimes D_k,$$
where $\gamma_k\in M_{m(p)}(\mathbb{C}),$ $1\leq k\leq p,$ are Pauli matrices.

\begin{lemma}\label{torus lemma} For every $x\in L_{\infty}(\mathbb{T}_{\theta}^p),$ we have
$${\rm Tr}(\pi(x)e^{-t^2D^2})=\frac{\pi^{\frac{p}{2}}m(p)}{t^p}\tau(x)+O(\frac1{t^{p-1}}).$$
\end{lemma}
\begin{proof} Let $\{u_{\mathbf{k}}\}_{\mathbf{k}\in\mathbb{Z}}$ be the standard basis in $L_2(\mathbb{T}^p_{\Theta},\tau),$ that is
$$u_{{\bf k}}=U_1^{k_1}U_2^{k_2}\cdots U_p^{k_p},\quad {\bf k}=(k_1,\cdots,k_p)\in\mathbb{Z}^p.$$
If $\{e_m\}_{m=1}^{m(p)}$ is the standard unit basis in $\mathbb{C}^{m(p)},$ then the elements $e_m\otimes u_{{\bf k}},$ $1\leq m\leq m(p),$ ${\bf k}\in\mathbb{Z}^p$ form an orthonormal basis in $\mathbb{C}^{m(p)}\otimes L_2(\mathbb{T}^p_{\Theta}).$ We have $(D^2)(e_m\otimes u_{\mathbf{k}})=|\mathbf{k}|^2e_m\otimes u_{\mathbf{k}}$ and $\langle \pi(x)(e_m\otimes u_{\mathbf{k}}),e_m\otimes u_{\mathbf{k}}\rangle=\tau(x)$ for every $x\in L_{\infty}(\mathbb{T}_{\theta}^p)$ and for every $1\leq m\leq p,$ $\mathbf{k}\in\mathbb{Z}^p.$ Hence,
$${\rm Tr}(\pi(x)e^{-t^2D^2})=\sum_{m=1}^{m(p)}\sum_{\mathbf{k}\in\mathbb{Z}^p}\langle\pi(x)e^{-t^2D^2}(e_m\otimes u_{\mathbf{k}}),e_m\otimes u_{\mathbf{k}}\rangle=$$
$$=m(p)\tau(x)\sum_{\mathbf{k}\in\mathbb{Z}^p}e^{-t^2|\mathbf{k}|^2}=m(p)\tau(x)(\sum_{k\in\mathbb{Z}}e^{-t^2k^2})^p.$$
The assertion follows now from Lemma \ref{pois lemma} \eqref{pois torus}.
\end{proof}

\begin{proof}[Proof of Corollary \ref{torus cor}] By Lemma \ref{torus lemma}, the noncommutative torus satisfies the Condition \ref{former regular}. The assertion follows now from Theorem \ref{fubini}.
\end{proof}

\subsection{Example: quantum group SU$_q(2)$}

In what follows, $\mathcal{O}(SU_q(2))$ is the algebraic linear span of all words in $a,c,a^*,c^*$ with the following cancellation rules\footnote{Here, we are using the notations from \cite{KSh}. The definition appears on p.102 in \cite{KSh}. The same definition is used by Connes (see formula (18) in \cite{Connes_suq2}) and Chakraborty-Pal (see p.2 in \cite{ch_pal}). However, those authors use the notation $\alpha$ for $a$ and $\beta$ for $c.$}
$$a^*a+c^*c=1,\ aa^*+q^2cc^*=1,\ ac=qca,\ ac^*=qc^*a,\ cc^*=c^*c.$$
Here, the parameter $q$ takes value from the interval $[-1,1].$ For $q=1,$ the algebra $\mathcal{O}(SU_q(2))$ is commutative and (its von Neumann envelope) equals to $L_{\infty}(SU(2)).$ In what follows, we assume $q\in(-1,1).$

It follows from Proposition IV.4 in \cite{KSh} that the elements
$$\{a^nc^m(c^*)^r,c^m(c^*)^r(a^*)^{n+1}\}_{m,n,r\in\mathbb{Z}_+}$$
form a Hamel basis in $\mathcal{O}(SU_q(2)).$ Define a linear functional\footnote{This is a trace on $\mathcal{O}(SU_q(2)),$ not the Haar state. However, we don't need the tracial property of $\tau.$ Connes (see Theorem 4 in \cite{Connes_suq2}) considered this functional on a larger algebra.} $\tau$ on $\mathcal{O}(SU_q(2))$ by setting
$$\tau(a^nc^m(c^*)^r)=\tau(c^m(c^*)^r(a^*)^n)=0,\quad (m,n,r)\neq(0,0,0),\quad \tau(1)=1.$$

The algebra $\mathcal{O}(SU_q(2))$ acts on the Hilbert space $H$ which is the completion of $\mathcal{O}(SU_q(2))$ with respect to the inner product $(x,y)\to h(xy^*).$ Here, $h$ is the Haar state on SU$_q(2)$ (defined in Theorem IV.14 in \cite{KSh}).

Let $l,m,n\in\frac12\mathbb{Z}_+$ be such that $l\geq0,$ $|m|,|n|\leq l$ and $l-m,l-n\in\mathbb{Z}.$ Let $t_{m,n}^l$ be the orthonormal basis in $\mathcal{O}(SU_q(2))$ constructed in Theorem IV.13 in \cite{KSh}. By construction of the Hilbert space $H,$ these elements also form an orthonormal basis in $H.$ Connes defined Dirac operator $D$ on the Hilbert space $H$ in \cite{Connes_suq2} (see formula (22) on p.8 there) by the formula
$$Dt_{m,n}^l=2l(2\delta_0(l-m)-1)t_{m,n}^l.$$
In this text, we are not interested in the sign of $D,$ but only in its absolute value given by the formula
$$|D|t_{m,n}^l=2lt_{m,n}^l.$$

\begin{lemma}\label{suq2 ok} For every $x\in \mathcal{O}(SU_q(2)),$ we have
$${\rm Tr}(M_xe^{-(tD)^2})=\frac{\pi^{\frac12}}{4t^3}\tau(x)+O(t^{-2}),\quad t\to0.$$
\end{lemma}
\begin{proof} {\bf Step 1:} Let $j\in\frac12\mathbb{Z}_+$ and let $(r,s)\neq(0,0).$ We claim that
$${\rm Tr}(M_{t^j_{r,s}}e^{-(tD)^2})=0.$$

Let $\mathcal{A}[\cdot,\cdot]$ be the linear subspace in $\mathcal{O}(SU_q(2))$ defined on p.105 in \cite{KSh}. By Lemma IV.11 in \cite{KSh}, we have $t^j_{r,s}\in\mathcal{A}[-2r,-2s]$ and $t^l_{m,n}\in\mathcal{A}[-2m,-2n].$ Using formula (26) on p.105 in \cite{KSh}, we infer that $t^j_{r,s}t^l_{m,n}\in\mathcal{A}[-2r-2m,-2s-2n].$ It follows now from formulae (47) and (56) in \cite{KSh} that $t^j_{r,s}t^l_{m,n}$ is orthogonal to $t^l_{m,n}$ (because $(r,s)\neq(0,0)$). Thus, 
$${\rm Tr}(M_{t^j_{r,s}}e^{-(tD)^2})=\sum_{l,m,n}e^{-4t^2l^2}\langle t^j_{r,s}t^l_{m,n},t^l_{m,n}\rangle=0.$$

{\bf Step 2:} Let $j\in\mathbb{Z}_+.$ We claim that
$${\rm Tr}(M_{(cc^*)^j}e^{-(tD)^2})=O(t^{-2}).$$

In what follows, $p_l=E_{|D|}\{2l\}.$ Using formulae (2.1)--(2.9) in \cite{ch_pal} (or formulae (19)--(21) in \cite{Connes_suq2}), we infer that
$$(p_lM_{cc^*}p_l)t^l_{m,n}=c(m,n,l)t^l_{m,n},$$
where
$$c(m,n,l)=O(q^{m+l})+O(q^{n+l}).$$
Taking into account that $|q|<1,$ we obtain
$${\rm Tr}(p_lM_{cc^*}p_l)=\sum_{m,n=-l}^lc(m,n,l)=\sum_{m,n=-l}^l\Big(O(q^{m+l})+O(q^{n+l})\Big)=O(l).$$

It follows that
$${\rm Tr}(M_{(cc^*)^j}p_l)={\rm Tr}(p_lM_{(cc^*)^j}p_l)\leq \|b\|_{\infty}^{2j-2}{\rm Tr}(p_lM_{cc^*}p_l)=O(l).$$
Hence,
$${\rm Tr}(M_{(cc^*)^j}e^{-(tD)^2})=\sum_{l\in\frac12\mathbb{Z}_+}e^{-4t^2l^2}{\rm Tr}(M_{(cc^*)^j}p_l)=\sum_{l\in\frac12\mathbb{Z}_+}e^{-4t^2l^2}\cdot O(l).$$
Thus,
$$|{\rm Tr}(M_{(cc^*)^j}e^{-(tD)^2})|\leq O(1)\cdot\Big(\sum_{l\in\frac12\mathbb{Z}_+}le^{-4l^2t^2}\Big)=O(1)\cdot\Big(\sum_{l\in\mathbb{Z}_+}le^{-l^2t^2}\Big).$$
Replacing the sum with an integral, we conclude the proof in Step 2.

{\bf Step 3:} Let $r,s\in\frac12\mathbb{Z}$ be such that $(r,s)\neq(0,0).$ By Step 1 and formula (47) in \cite{KSh}, we have
$${\rm Tr}(M_xe^{-(tD)^2})=0,\quad x\in\mathcal{A}[-2r,-2s].$$
By Step 2 and Proposition 10 (i) in \cite{KSh}, we have
$${\rm Tr}(M_xe^{-(tD)^2})=O(t^{-2}),\quad x\in\mathcal{A}[0,0],\ \tau(x)=0.$$
A trivial computation shows that
$${\rm Tr}(e^{-(tD)^2})=\sum_{l\in\frac12\mathbb{Z}_+}(2l+1)^2e^{-4l^2t^2}=O(t^{-2})+\frac12\sum_{l\in\mathbb{Z}}l^2e^{-l^2t^2}.$$
It follows from Lemma \ref{pois lemma} \eqref{pois suq2} that
$${\rm Tr}(e^{-(tD)^2})=\frac{\pi^{\frac12}}{4t^3}+O(t^{-2}).$$
Thus, for every $m,n\in\mathbb{Z},$ we have
$${\rm Tr}(M_xe^{-(tD)^2})=\frac{\pi^{\frac12}}{4t^3}\tau(x)+O(t^{-2}),\quad x\in\mathcal{A}[m,n].$$
The assertion follows now from formula (32) in \cite{KSh}.
\end{proof}

\begin{proof}[Proof of Corollary \ref{sphere suq2 corollary}] By Lemma \ref{sphere ok} and Lemma \ref{suq2 ok}, spectral triples corresponding to the sphere $\mathbb{S}^2$ and to the quantum group SU$_q(2)$ satisfy the Condition \ref{former regular}. The assertion follows now from Theorem \ref{fubini}. 
\end{proof}

\section{Proof of Theorem \ref{hard example}}

The following lemma provides a convenient formula for the sum of the first $n$ eigenvalues of ${\rm diag}(x)\in\mathcal{L}_{1,\infty}.$

\begin{lemma}\label{kalton} If $x\in l_{\infty}$ is such that $|x(k)|\leq\frac1{k+1},$ $k\geq0,$ then
$$\sum_{m=0}^n\lambda(m,x)=\sum_{k=0}^nx(k)+O(1).$$
\end{lemma}
\begin{proof} Suppose first that $x\geq0.$ It is clear that
$$\sum_{k=0}^nx(k)\leq\sum_{k=0}^n\mu(k,x).$$
On the other hand, there exists a set $A_n\subset\mathbb{Z}_+$ such that $|A|=n+1$ and such that
$$\sum_{k=0}^n\mu(k,x)=\sum_{k\in A_n}x(k)=\sum_{k\in A_n,k\leq n}x(k)+\sum_{k\in A_n,k\geq n}x(k)\leq$$
$$\leq\sum_{k=0}^nx(k)+\sum_{k\in A_n,k\geq n}\frac1{k+1}\leq\sum_{k=0}^nx(k)+\sum_{k=n+1}^{2n+1}\frac1{k+1}\leq\sum_{k=0}^nx(k)+1.$$
A combination of the latter estimates yields the assertion under the additional assumption that $x\geq0.$

For an arbitrary $x\in l_{1,\infty},$ there exist $0\leq x_p\in l_{1,\infty},$ $1\leq p\leq 4,$ such that
$$x=x_1+ix_2+i^2x_3+i^3x_4.$$
If $|x(k)|\leq\frac1{k+1},$ $k\geq0,$ then also $x_p(k)\leq\frac1{k+1},$ $k\geq0.$ It follows from Lemma 5.7.5 in \cite{LSZ} that
$$\sum_{m=0}^n\lambda(m,x)=\sum_{p=1}^4i^{p-1}\sum_{m=0}^n\lambda(m,x_p)+O(1).$$
Applying the assertion for positive operators $x_p,$ $1\leq p\leq 4,$ we infer that
$$\sum_{m=0}^n\lambda(m,x)=\sum_{p=1}^4i^{p-1}\sum_{k=0}^nx_p(k)+O(1)=\sum_{k=0}^nx(k)+O(1).$$
This concludes the proof.
\end{proof}

\begin{lemma}\label{double kalton} If $x\in l_{\infty}(\mathbb{Z}^2)$ is such that $|x(k,l)|\leq\frac1{1+k^2+l^2},$ $k,l\geq0,$ then
$$\sum_{m=0}^n\lambda(m,x)=\sum_{k,l=0}^{\floor{n^{1/2}}}x(k,l)+O(1).$$
\end{lemma}
\begin{proof} Define a bijection $\alpha_2:\mathbb{Z}_+\to\mathbb{Z}_+^2$ as in Lemma \ref{nte2}. Define $z\in l_{\infty}$ by setting $z=x\circ\alpha_2.$  It follows from Lemma \ref{nte2} that
$$|z(m)|\leq\frac1{1+|\alpha_2(m)|^2}\leq\frac{{\rm const}}{m+1},\quad m\geq0.$$
Therefore,
$$\sum_{m=0}^n\lambda(m,x)=\sum_{m=0}^n\lambda(m,z)\stackrel{L.\ref{kalton}}{=}\sum_{m=0}^nz(m)+O(1)=\sum_{m=0}^nx(\alpha_2(m))+O(1).$$
Note that
$$\Big|\sum_{m=0}^nx(\alpha_2(m))-\sum_{\substack{\mathbf{k}\in\mathbb{Z}_+^2\\|\mathbf{k}|\leq|\alpha_2(n)|}}x(\mathbf{k})\Big|\leq\sum_{\substack{\mathbf{k}\in\mathbb{Z}_+^2\\|\mathbf{k}|=|\alpha_2(n)|}}|x(\mathbf{k})|\leq\sum_{\substack{\mathbf{k}\in\mathbb{Z}_+^2\\|\mathbf{k}|=|\alpha_2(n)|}}\frac1{1+|\mathbf{k}|^2}=$$
$$=\frac1{1+|\alpha_2(n)|^2}\sum_{\substack{\mathbf{k}\in\mathbb{Z}_+^2\\|\mathbf{k}|=|\alpha_2(n)|}}1\leq\frac1{1+|\alpha_2(n)|^2}\sum_{k=0}^{|\alpha_2(n)|}1\leq 1.$$
It follows from Lemma \ref{nte2} that
$$\Big|\sum_{\substack{\mathbf{k}\in\mathbb{Z}_+^2\\|\mathbf{k}|\leq|\alpha_2(n)|}}x(\mathbf{k})-\sum_{\substack{\mathbf{k}\in\mathbb{Z}_+^2\\|\mathbf{k}|^2\leq n}}x(\mathbf{k})\Big|\leq\sum_{\substack{\mathbf{k}\in\mathbb{Z}_+^2\\|\mathbf{k}|^2\in[|\alpha_2(n)|^2,n]}}|x(\mathbf{k})|\leq\sum_{\substack{\mathbf{k}\in\mathbb{Z}_+^2\\|\mathbf{k}|^2\in[c_1n,c_2n]}}\frac1{1+|\mathbf{k}|^2}=O(1).$$
We also have
$$\Big|\sum_{k_1^2+k_2^2\leq n}x(k_1,k_2)-\sum_{0\leq k_1,k_2\leq n^{1/2}}x(k_1,k_2)\Big|\leq\sum_{\substack{0\leq k_1,k_2\leq n^{1/2}\\ k_1^2+k_2^2>n}}|x(k_1,k_2)|\leq$$
$$\leq\sum_{\substack{0\leq k_1,k_2\leq n^{1/2}\\ k_1^2+k_2^2>n}}\frac1{1+k_1^2+k_2^2}\leq\frac1{1+n}\sum_{0\leq k_1,k_2\leq n^{1/2}}1=O(1).$$
A combination of the latter estimates yields the assertion.
\end{proof}

For a given $\theta\in\mathbb{R},$ we define $x_{\theta}\in l_{\infty}$ by setting $x_{\theta}(0)=1$ and
$$x_{\theta}(k)=e^{in\theta},\quad k\in[2^n,2^{n+1}),\quad n\geq0.$$
Let $D={\rm diag}(\{k\}_{k\geq0}),$ $U_{\theta}={\rm diag}(\{x_{\theta}(k)\}_{k\geq0}).$ Clearly, $(1+D^2)^{-1/2}\in\mathcal{L}_{1,\infty}.$

\begin{lemma}\label{degt lemma} For every $\theta\notin 2\pi\mathbb{Z},$ we have $U_{\theta}D^{-1}\in[\mathcal{L}_{1,\infty},\mathcal{L}(H)].$
\end{lemma}
\begin{proof} We have
$$\sum_{k=1}^{2^{n+1}-1}\frac{x_{\theta}(k)}{k}=\sum_{m=0}^ne^{im\theta}\sum_{k=2^m}^{2^{m+1}-1}\frac1k=\sum_{m=0}^ne^{im\theta}(\log(2)+O(2^{-m}))=$$
$$=O(1)+\log(2)\frac{e^{i(n+1)\theta}-e^{i\theta}}{e^{i\theta}-1}=O(1).$$
Hence, for every $m\geq1,$ we have
$$\sum_{k=0}^m\frac{x_{\theta}(k)}{(1+k^2)^{1/2}}=O(1).$$
By Lemma \ref{kalton}, we have
$$\sum_{k=0}^m\lambda(U_{\theta}(1+D^2)^{-1/2})=O(1).$$
The assertion follows from Theorem \ref{spectral}.
\end{proof}

\begin{lemma}\label{compute lemma} For every $m_1,m_2\in\mathbb{Z}_+,$ we have
$$\sum_{k=2^{m_1}}^{2^{m_1+1}-1}\sum_{l=2^{m_2}}^{2^{m_2+1}-1}\frac1{k^2+l^2}=\Xi(|m_1-m_2|)+O(\min\{2^{-m_1},2^{-m_2}\}).$$
Here,
\begin{equation}\label{alfa def}
\Xi(m)=\int_1^2\int_{2^m}^{2^{m+1}}\frac{dtds}{t^2+s^2},\quad m\in\mathbb{Z}.
\end{equation}
\end{lemma}
\begin{proof} It is clear that
$$\int_{2^{m_2}}^{2^{m_2+1}}\frac{dt}{t^2+k^2}\leq\sum_{l=2^{m_2}}^{2^{m_2+1}-1}\frac{1}{k^2+l^2}\leq \int_{2^{m_2}}^{2^{m_2+1}}\frac{dt}{t^2+k^2}+\frac{1}{k^2+2^{2m_2}}.$$
Thus,
$$\sum_{l=2^{m_2}}^{2^{m_2+1}-1}\frac{1}{k^2+l^2}=\int_{2^{m_2}}^{2^{m_2+1}}\frac{dt}{t^2+k^2}+\frac{O(1)}{k^2+2^{2m_2}}.$$
It follows that
$$\sum_{k=2^{m_1}}^{2^{m_1+1}-1}\sum_{l=2^{m_2}}^{2^{m_2+1}-1}\frac{1}{k^2+l^2}=\int_{2^{m_2}}^{2^{m_2+1}}\sum_{k=2^{m_1}}^{2^{m_1+1}-1}\frac{dt}{t^2+k^2}+O(1)\cdot\sum_{k=2^{m_1}}^{2^{m_1+1}-1}\frac1{k^2+2^{2m_2}}.$$
Repeating the argument, we obtain that
$$\sum_{k=2^{m_1}}^{2^{m_1+1}-1}\sum_{l=2^{m_2}}^{2^{m_2+1}-1}\frac{1}{k^2+l^2}=\int_{2^{m_2}}^{2^{m_2+1}}\int_{2^{m_1}}^{2^{m_1+1}}\frac{dtds}{t^2+s^2}+$$
$$+O(1)\cdot\int_{2^{m_2}}^{2^{m_2+1}}\frac{dt}{t^2+2^{2m_1}}+O(1)\cdot\int_{2^{m_1}}^{2^{m_1+1}}\frac{ds}{s^2+2^{2m_2}}+\frac{O(1)}{2^{2m_1}+2^{2m_2}}.$$
Clearly, the second and third integrals above can be estimated as
$$O(1)\cdot\frac{2^{m_2}}{2^{2m_1}+2^{2m_2}},\quad O(1)\cdot\frac{2^{m_1}}{2^{2m_1}+2^{2m_2}}.$$
The reference to \eqref{alfa def} concludes the proof.
\end{proof}

It is obvious that
\begin{equation}\label{exp est}
0\leq\Xi(m)\leq 2^{-m},\quad m\in\mathbb{Z}_+.
\end{equation}

\begin{lemma}\label{key lemma} For every $\theta\in\mathbb{R}$ and for every $p\in\mathbb{Z},$ we have
$$\sum_{k,l=1}^M\frac{x_{\theta}^p(k)x_{\theta}^{-p}(l)}{k^2+l^2}=F(p\theta)\frac{\log(M)}{\log(2)}+O(1),\quad M\in\mathbb{N}.$$
Here,
\begin{equation}\label{F def}
F(\theta)=\sum_{m\in\mathbb{Z}}\Xi(|m|)e^{im\theta},
\end{equation}
where $\Xi$ is given in \eqref{alfa def}.
\end{lemma}
\begin{proof} Since $x_{\theta}^p=x_{p\theta},$ it follows that we may consider only the case $p=1.$ Firstly, we establish the assertion for $M=2^{n+1}-1,$ $n\in\mathbb{Z}_+.$ It follows from Lemma \ref{compute lemma} that
$$\sum_{k,l=1}^{2^{n+1}-1}\frac{x_{\theta}(k)x_{\theta}^{-1}(l)}{k^2+l^2}=\sum_{m_1,m_2=0}^ne^{i(m_1-m_2)\theta}\sum_{k=2^{m_1}}^{2^{m_1+1}-1}\sum_{l=2^{m_2}}^{2^{m_2+1}-1}\frac1{k^2+l^2}=$$
$$=\sum_{m_1,m_2=0}^n\Xi(|m_1-m_2|)e^{i(m_1-m_2)\theta}+O(1).$$
Rearranging the summands, we obtain that
$$\sum_{m_1,m_2=0}^n\Xi(|m_1-m_2|)e^{i(m_1-m_2)\theta}=$$
$$=(n+1)\Xi(0)+\sum_{m=1}^n(n+1-m)e^{im\theta}\Xi(m)+\sum_{m=1}^n(n+1-m)e^{-im\theta}\Xi(m).$$
It follows from \eqref{exp est} that
$$\sum_{m_1,m_2=0}^n\Xi(|m_1-m_2|)e^{i(m_1-m_2)}=n\Big(\sum_{m=-n}^n\Xi(|m|)e^{im\theta}\Big)+O(1)=nF(\theta)+O(1).$$
This proves the assertion for $M=2^{n+1}-1,$ $n\in\mathbb{Z}_+.$

Now, for an arbitrary $M\in[2^n,2^{n+1}),$ we have
$$|\sum_{k,l=1}^{2^{n+1}-1}\frac{x_{\theta}(k)x_{\theta}^{-1}(l)}{k^2+l^2}-\sum_{k,l=1}^M\frac{x_{\theta}(k)x_{\theta}^{-1}(l)}{k^2+l^2}|\leq\sum_{k,l=1}^{2^{n+1}-1}\frac1{k^2+l^2}-\sum_{k,l=1}^{2^n-1}\frac1{k^2+l^2}\leq$$
$$\leq 2\sum_{k=2^n}^{2^{n+1}-1}\sum_{l=1}^{2^{n+1}-1}\frac1{k^2+l^2}\leq2\sum_{k=2^n}^{2^{n+1}-1}\sum_{l=1}^{2^{n+1}-1}2^{-2n}\leq 4.$$
This concludes the proof.
\end{proof}

The proof of the following lemma is parallel (though, not identical) to that of Lemma \ref{key lemma}.

\begin{lemma}\label{second key lemma} For every $\theta\in\mathbb{R}$ and for every $p,q\in\mathbb{Z}$ such that $(p+q)\theta\notin2\pi\mathbb{Z},$ we have
$$\sum_{k,l=1}^M\frac{x_{\theta}^p(k)x_{\theta}^q(l)}{k^2+l^2}=O(1),\quad M\in\mathbb{N}.$$
\end{lemma}
\begin{proof} Firstly, we establish the assertion for $M=2^{n+1}-1,$ $n\in\mathbb{Z}_+.$ It follows from Lemma \ref{compute lemma} that
$$\sum_{k,l=1}^{2^{n+1}-1}\frac{x_{\theta}^p(k)x_{\theta}^q(l)}{k^2+l^2}=\sum_{m_1,m_2=0}^ne^{i(pm_1+qm_2)\theta}\sum_{k=2^{m_1}}^{2^{m_1+1}-1}\sum_{l=2^{m_2}}^{2^{m_2+1}-1}\frac1{k^2+l^2}=$$
$$=\sum_{m_1,m_2=0}^n\Xi(|m_1-m_2|)e^{ip(m_1-m_2)\theta}e^{i(p+q)m_2\theta}+O(1).$$
Rearranging the summands, we obtain that
$$\sum_{m_1,m_2=0}^n\Xi(|m_1-m_2|)e^{ip(m_1-m_2)\theta}e^{i(p+q)m_2\theta}=\sum_{m=0}^n\Xi(0)e^{i(p+q)m\theta}+$$
$$+\sum_{m=1}^n\Xi(m)e^{ipm\theta}\sum_{m_2=0}^{n-m}e^{i(p+q)m_2\theta}+\sum_{m=1}^n\Xi(m)e^{-ipm\theta}\sum_{m_2=m}^ne^{i(p+q)m_2\theta}.$$
The assumption $(p+q)\theta\notin2\pi\mathbb{Z}$ guarantees that
$$\sum_{m_2=0}^{n-m}e^{i(p+q)m_2\theta}=O(1),\quad \sum_{m_2=m}^ne^{i(p+q)m_2\theta}=O(1),\quad m,n\in\mathbb{Z}_+.$$
Therefore, appealing to \eqref{exp est}, we obtain
$$|\sum_{m_1,m_2=0}^n\Xi(|m_1-m_2|)e^{ip(m_1-m_2)\theta}e^{i(p+q)m_2\theta}|\leq|\sum_{m=0}^n\Xi(0)e^{i(p+q)m\theta}|+$$
$$+\sum_{m=1}^n\Xi(m)\cdot O(1)+\sum_{m=1}^n\Xi(m)\cdot O(1)=O(1).$$
In other words, we have
$$\sum_{k,l=1}^{2^{n+1}-1}\frac{x_{\theta}^p(k)x_{\theta}^q(l)}{k^2+l^2}=O(1).$$
This proves the assertion for $M=2^{n+1}-1,$ $n\in\mathbb{Z}_+.$

Now, for an arbitrary $M\in[2^n,2^{n+1}),$ we have
$$|\sum_{k,l=1}^{2^{n+1}-1}\frac{x_{\theta}^p(k)x_{\theta}^q(l)}{k^2+l^2}-\sum_{k,l=1}^M\frac{x_{\theta}^p(k)x_{\theta}^q(l)}{k^2+l^2}|\leq\sum_{k,l=1}^{2^{n+1}-1}\frac1{k^2+l^2}-\sum_{k,l=1}^{2^n-1}\frac1{k^2+l^2}\leq$$
$$\leq 2\sum_{k=2^n}^{2^{n+1}-1}\sum_{l=1}^{2^{n+1}-1}\frac1{k^2+l^2}\leq2\sum_{k=2^n}^{2^{n+1}-1}\sum_{l=1}^{2^{n+1}-1}2^{-2n}\leq 4.$$
This concludes the proof.
\end{proof}

\begin{proof}[Proof of Theorem \ref{hard example}] Let $F$ be as in \eqref{F def}. Fourier coefficients of $F$ are given by a non-zero sequence $\{\Xi(|m|)\}_{m\in\mathbb{Z}}$ and, therefore $F\neq0.$ It follows from \eqref{exp est} that Fourier series for $F$ converges uniformly and, therefore, $F$ is continuous. It follows from the continuity of $F$ that one can choose $\theta$ such that $\frac{\theta}{2\pi}\in\mathbb{Q},$ $\theta\notin2\pi\mathbb{Z}$ and such that $F(\theta)\neq0.$ Let $\mathcal{A}_{\theta}$ be the von Neumann subalgebra in $\mathcal{L}(l_2)$ generated by $U_{\theta}.$

Since $\theta\in 2\pi\mathbb{Q},$ it follows that there exists $0\neq r\in\mathbb{Z}$ such that $U_{\theta}^r=1$ and, therefore, $\mathcal{A}_{\theta}$ is finite dimensional. Every linear functional on a finite dimensional subalgebra in $\mathcal{L}(H)$ is automatically normal. It follows that the mapping
$$T\to\varphi(T(1+D^2)^{-1/2}),\quad T\in\mathcal{A}_{\theta}$$
is normal for every linear functional on $\mathcal{L}_{1,\infty}$ (in particular, for every trace on $\mathcal{L}_{1,\infty}$). It follows from Lemma \ref{degt lemma} that, for every $p\in\mathbb{Z}$ and for every normalised trace $\varphi$ on $\mathcal{L}_{1,\infty},$ we have
$$\varphi(U_{\theta}^p(1+D^2)^{-1/2})=
\begin{cases}
1,\quad p\theta\in 2\pi\mathbb{Z}\\
0,\quad p\theta\notin 2\pi\mathbb{Z}.
\end{cases}
$$
Hence, $T(1+D^2)^{-1/2}$ is universally measurable for every $T\in\mathcal{A}_{\theta}.$ This proves \eqref{hardc}.

Since $\mathcal{A}_{\theta}\otimes\mathcal{A}_{\theta}$ is also finite dimensional, it follows that the mapping
$$T\to\varphi(T(1+D^2\otimes 1+1\otimes D^2)^{-1/2}),\quad T\in\mathcal{A}_{\theta}\otimes\mathcal{A}_{\theta}$$
is automatically normal for every linear functional on $\mathcal{L}_{1,\infty}$ (in particular, for every trace on $\mathcal{L}_{1,\infty}$).

It follows from Lemma \ref{key lemma} that, for every $\theta\in\mathbb{R}$ and for every $p\in\mathbb{Z},$ we have
$$\sum_{k,l=0}^M\frac{x_{\theta}^p(k)x_{\theta}^{-p}(l)}{1+k^2+l^2}=F(p\theta)\frac{\log(M+1)}{\log(2)}+O(1),\quad M\in\mathbb{Z}_+.$$
This equality combined with Lemma \ref{double kalton} provides that
$$\sum_{m=0}^N\lambda(m,(U_{\theta}^p\otimes U_{\theta}^{-p})(1+1\otimes D^2+D^2\otimes 1)^{-1})=F(p\theta)\frac{\log(N+1)}{2\log(2)}+O(1),\quad N\in\mathbb{Z}.$$
By Theorem \ref{spectral}, we have that
\begin{equation}\label{peqmq}
\varphi((U_{\theta}^p\otimes U_{\theta}^{-p})(1+1\otimes D^2+D^2\otimes 1)^{-1})=\frac1{2\log(2)}F(p\theta)
\end{equation}
for every normalised trace $\varphi$ on $\mathcal{L}_{1,\infty}.$

It follows from Lemma \ref{second key lemma} that, for every $\theta\in\mathbb{R}$ and for every $p,q\in\mathbb{Z}$ such that $(p+q)\theta\notin2\pi\mathbb{Z},$ we have
$$\sum_{k,l=0}^M\frac{x_{\theta}^p(k)x_{\theta}^q(l)}{1+k^2+l^2}=O(1),\quad M\in\mathbb{Z}_+.$$
By Lemma \ref{double kalton}, we have that
$$\sum_{m=0}^N\lambda(m,(U_{\theta}^p\otimes U_{\theta}^q)(1+1\otimes D^2+D^2\otimes 1)^{-1})=O(1),\quad N\in\mathbb{Z}$$
for every $p,q\in\mathbb{Z}$ with $(p+q)\theta\notin2\pi\mathbb{Z}.$ By Theorem \ref{spectral}, we have
\begin{equation}\label{pneqq}
\varphi((U_{\theta}^p\otimes U_{\theta}^q)(1+1\otimes D^2+D^2\otimes 1)^{-1})=0
\end{equation}
for every normalised trace $\varphi$ on $\mathcal{L}_{1,\infty}.$

Combining \eqref{peqmq} and \eqref{pneqq}, we conclude that elements of the form
$$T(1+1\otimes D^2+D^2\otimes 1)^{-1},\quad T\in\mathcal{A}_{\theta}\otimes\mathcal{A}_{\theta},$$
are universally measurable. This proves \eqref{hardd}.

Finally, the first assertion in \eqref{harde} follows from \eqref{peqmq} (for $p=1$) and the second assertion in \eqref{harde} follows from Lemma \ref{degt lemma}.
\end{proof}

\begin{proof}[Proof of Corollary \eqref{ex for pos}] Set $T=U+U^{-1}+2\geq0.$ In the course of the proof of Theorem \ref{hard example}, we established a formula \eqref{pneqq}, which implies
$$\varphi((U\otimes U)(1+1\otimes D^2+D^2\otimes 1)^{-1})=\varphi((U^{-1}\otimes U^{-1})(1+1\otimes D^2+D^2\otimes 1)^{-1})=0.$$
It follows from Theorem \ref{hard example} \eqref{harde} that
$$\varphi((T\otimes T)(1+1\otimes D^2+D^2\otimes 1)^{-1})=4\varphi((1+1\otimes D^2+D^2\otimes 1)^{-1})+$$
$$+2\varphi((U\otimes U^{-1})(1+1\otimes D^2+D^2\otimes 1)^{-1})\neq 4\varphi((1+1\otimes D^2+D^2\otimes 1)^{-1}).$$
On the other hand, it follows from Lemma \ref{nte3} and Theorem \ref{hard example} \eqref{harde} that
$$4\varphi((1+1\otimes D^2+D^2\otimes 1)^{-1})=\pi=\frac{\pi}{4}(\varphi(T(1+D^2)^{-1/2}))^2.$$
\end{proof}

\begin{remark} Neither Theorem \ref{hard example} nor its proof specifies the dimension of the algebra $\mathcal{A}.$ However, if we replace $2$ with $2^7$ in the definition of $x_{\theta}$ and set $\theta=\pi,$ then the algebra $\mathcal{A}$ becomes $2-$dimensional. That $F(\pi)\neq0$ can be showed as in the proof of Theorem \ref{second example} below.
\end{remark}

\section{Proof of Theorem \ref{second example}}

Define the sequence $d$ by setting $d(0)=0,$
$$d(k)=
\begin{cases}
k,& k\in[2^{7n},2^{7(n+1)}),\quad n=0 \ {\rm mod} \ 2\\
2^7k,& k\in[2^{7n},2^{7(n+1)}),\quad n=1 \ {\rm mod} \ 2
\end{cases}
$$
and set $D={\rm diag}(\{d(k)\}_{k\geq0}).$

Set
$$\Xi_0(m)=\int_1^{2^7}\int_{2^{7m}}^{2^{7(m+1)}}\frac{dtds}{t^2+s^2}.$$
The proof of the following lemma is identical to that of Lemma \ref{compute lemma} and is, therefore, omitted.

\begin{lemma}\label{third compute} We have
$$\sum_{k_1=2^{7n_1}}^{2^{7(n_1+1)}-1}\sum_{k_2=2^{7n_2}}^{2^{7(n_2+1)}-1}\frac1{d^2(k_1)+d^2(k_2)}=$$
$$=O(\min\{2^{-7n_1},2^{-7n_2}\})+
\begin{cases}
\Xi_0(n_2-n_1),& n_1=0 \ {\rm mod} \ 2, n_2=0 \ {\rm mod} \ 2\\
2^{-7}\Xi_0(n_2-n_1-1),& n_1=1 \ {\rm mod} \ 2,n_2=0 \ {\rm mod} \ 2\\
2^{-7}\Xi_0(n_2-n_1+1),& n_1=0 \ {\rm mod} \ 2,n_2=1 \ {\rm mod} \ 2\\
2^{-14}\Xi_0(n_2-n_1),& n_1=1 \ {\rm mod} \ 2,n_2=1 \ {\rm mod} \ 2
\end{cases}
$$
\end{lemma}

Note that
$$\Xi_0(m)=\Xi_0(|m|)\leq\int_1^{2^7}\int_{2^{7|m|}}^{2^{7(|m|+1)}}\frac{dtds}{2^{14|m|}}\leq (2^7-1)^2\cdot2^{-7m}.$$
In particular, we have
$$\sum_{m\in\mathbb{Z}}\Xi_0(m)<\infty.$$

\begin{lemma} We have
$$\sum_{k_1,k_2=1}^M\frac1{d^2(k_1)+d^2(k_2)}=\frac{\log(M)}{14\log(2)}(1+2^{-7})^2\sum_{m\in\mathbb{Z}}\Xi_0(2m)+O(1).$$
\end{lemma}
\begin{proof} Suppose first that $M=2^{7(n+1)}-1.$ It follows from Lemma \ref{third compute} that
$$\sum_{k_1,k_2=1}^{2^{7(n+1)}-1}\frac1{d^2(k_1)+d^2(k_2)}=\sum_{\substack{0\leq n_1,n_2\leq n\\ n_1,n_2=0 \, {\rm mod} \, 2}}\Xi_0(n_2-n_1)+2^{-14}\sum_{\substack{0\leq n_1,n_2\leq n\\ n_1,n_2=1 \, {\rm mod} \, 2}}\Xi_0(n_2-n_1)+$$
$$+2^{-7}\sum_{\substack{0\leq n_1,n_2\leq n\\n_1=1 \, {\rm mod} \, 2\\n_2=0 \, {\rm mod} \, 2}}\Xi_0(n_2-n_1-1)+2^{-7}\sum_{\substack{0\leq n_1,n_2\leq n\\n_1=0 \, {\rm mod} \, 2\\n_2=1 \, {\rm mod} \, 2}}\Xi_0(n_2-n_1+1)+O(1).$$
Making the substitution
$$(n_1,n_2)=
\begin{cases}
(m_1,m_2),& n_1=0 \ {\rm mod} \ 2,n_2=0 \ {\rm mod} \ 2\\
(m_1-1,m_2),& n_1=1 \ {\rm mod} \ 2,n_2=0 \ {\rm mod} \ 2\\
(m_1,m_2-1),& n_1=0 \ {\rm mod} \ 2,n_2=1 \ {\rm mod} \ 2\\
(m_1-1,m_2-1),& n_1=1 \ {\rm mod} \ 2,n_2=1 \ {\rm mod} \ 2
\end{cases}
$$
we have that
$$\sum_{k_1,k_2=1}^{2^{7(n+1)}-1}\frac1{d^2(k_1)+d^2(k_2)}=(1+2^{-7})^2\sum_{\substack{0\leq m_1\leq n\\1\leq m_2\leq n\\m_1,m_2=0 \ {\rm mod} \ 2}}\Xi_0(m_2-m_1)+O(1).$$
Rearranging the summands as in Lemma \ref{key lemma}, we infer that
$$\sum_{\substack{0\leq m_1\leq n\\1\leq m_2\leq n\\m_1,m_2=0 \ {\rm mod} \ 2}}\Xi_0(m_2-m_1)=\frac{n}{2}\sum_{m\in\mathbb{Z}}\Xi_0(2m)+O(1).$$
Passing from $M=2^{7(n+1)}-1$ to generic $M$ as in Lemma \ref{key lemma}, we conclude the proof.
\end{proof}

\begin{lemma}\label{fourth compute} For every normalised trace $\varphi$ on $\mathcal{L}_{1,\infty},$ we have
$$\varphi((1+D^2\otimes 1+1\otimes D^2)^{-1})=\frac1{7\log(2)}(\frac{1+2^{-7}}{2})^2\sum_{m\in\mathbb{Z}}\Xi_0(2m),$$
$$\varphi((1+D^2)^{-1/2})=\frac{1+2^{-7}}{2}.$$
\end{lemma}
\begin{proof} It follows from Lemma \ref{third compute} that
$$\sum_{k_1,k_2=0}^M\frac1{1+d^2(k_1)+d^2(k_2)}=\frac{\log(M)}{14\log(2)}(1+2^{-7})^2\sum_{m\in\mathbb{Z}}\Xi_0(2m)+O(1).$$
It follows now from Lemma \ref{double kalton} that
$$\sum_{k=0}^M\lambda(k,(1+D^2\otimes 1+1\otimes D^2)^{-1})=\frac{\log(M)}{28\log(2)}(1+2^{-7})^2\sum_{m\in\mathbb{Z}}\Xi_0(2m)+O(1).$$
The first assertion follows now from Theorem \ref{spectral}.

The second assertion follows from the equality
$$\frac1{(1+d^2(k))^{1/2}}=O(k^{-2})+\begin{cases}
k^{-1},& k\in[2^{7n},2^{7(n+1)}),\quad n=0 \ {\rm mod} \ 2\\
2^{-7}k^{-1},& k\in[2^{7n},2^{7(n+1)}),\quad n=1 \ {\rm mod} \ 2
\end{cases}=$$
$$=O(k^{-2})+\frac{1+2^{-7}}{2k}+\frac{1-2^{-7}}{2k}\cdot\begin{cases}
1,& k\in[2^{7n},2^{7(n+1)}),\quad n=0 \ {\rm mod} \ 2\\
-1,& k\in[2^{7n},2^{7(n+1)}),\quad n=1 \ {\rm mod} \ 2.
\end{cases}$$
\end{proof}

\begin{proof}[Proof of Theorem \ref{second example}] According to the Lemma \ref{fourth compute}, it suffices to show that
$$\sum_{m\in\mathbb{Z}}\Xi_0(2m)>\frac{7\pi}{4}\log(2).$$
In fact, we have
$$\Xi_0(0)=\int_1^{2^7}\int_1^{2^7}\frac{dtds}{t^2+s^2}=2\int_{1\leq s\leq t\leq 2^7}\frac{dtds}{t^2+s^2}\geq2\int_{1\leq s\leq t\leq 2^7}\frac{dtds}{2t^2}=$$
$$=\int_1^{2^7}\frac{(t-1)dt}{t^2}> 7\log(2)-1.$$
Therefore,
$$\sum_{m\in\mathbb{Z}}\Xi_0(2m)>\Xi_0(0)>7\log(2)-1>\frac{7\pi}{4}\log(2).$$
\end{proof}

\appendix

\section{Number-theoretic estimates}

The following lemmas are standard in number theory.

\begin{lemma}\label{nte1} For every $p\in\mathbb{N},$ we have
$$\sum_{\substack{\mathbf{k}\in\mathbb{Z}_+^p\\|\mathbf{k}|\leq m}}1=\frac{2^{-p}\pi^{\frac{p}{2}}}{\Gamma(1+\frac{p}{2})}m^p+O(m^{p-1}),\quad \sum_{\substack{\mathbf{k}\in\mathbb{Z}^p\\|\mathbf{k}|\leq m}}1=\frac{\pi^{\frac{p}{2}}}{\Gamma(1+\frac{p}{2})}m^p+O(m^{p-1}).$$
\end{lemma}
\begin{proof} We prove the first assertion by induction on $p.$ Let $K=[0,1]^p$ be the unit cube. For brevity, we denote $p-$tuple $(1,\cdots,1)=1.$ We have
$$\sum_{\substack{\mathbf{k}\in\mathbb{N}^p\\|\mathbf{k}|\leq m}}1=\sum_{\substack{\mathbf{k}\in\mathbb{N}^p\\|\mathbf{k}|\leq m}}\int_{K+\mathbf{k}-1}dt=\int_{C_m}dt,\quad\sum_{\substack{\mathbf{k}\in\mathbb{Z}_+^p\\|\mathbf{k}|\leq m}}1=\sum_{\substack{\mathbf{k}\in\mathbb{Z}_+^p\\|\mathbf{k}|\leq m}}\int_{K+\mathbf{k}}dt=\int_{B_m}dt,$$
where
$$C_m=\bigcup_{\substack{\mathbf{k}\in\mathbb{N}^p\\|\mathbf{k}|\leq m}}(K+\mathbf{k}-1)\subset\{t\in\mathbb{R}_+^p:\ |t|\leq m\}\subset\bigcup_{\substack{\mathbf{k}\in\mathbb{Z}_+^p\\|\mathbf{k}|\leq m}}(K+\mathbf{k})=B_m.$$
It follows immediately that
\begin{equation}\label{summa po sharu}
\sum_{\substack{\mathbf{k}\in\mathbb{N}^p\\|\mathbf{k}|\leq m}}1\leq\int_{\substack{t\in\mathbb{R}_+^p\\|t|\leq m}}dt\leq\sum_{\substack{\mathbf{k}\in\mathbb{Z}_+^p\\|\mathbf{k}|\leq m}}1.
\end{equation}

It is clear that
\begin{equation}\label{left right}
\sum_{\substack{\mathbf{k}\in\mathbb{Z}_+^p\\|\mathbf{k}|\leq m}}1-\sum_{\substack{\mathbf{k}\in\mathbb{N}^p\\|\mathbf{k}|\leq m}}1\leq p\sum_{\substack{\mathbf{k}\in\mathbb{Z}_+^{p-1}\\|\mathbf{k}|\leq m}}1=O(m^{p-1}),
\end{equation}
where we used induction with respect to $p$ in the last equality. Combining \eqref{summa po sharu} and \eqref{left right}, we infer that
$$\sum_{\substack{\mathbf{k}\in\mathbb{Z}_+^p\\|\mathbf{k}|\leq m}}1=\int_{\substack{t\in\mathbb{R}_+^p\\|t|\leq m}}dt+O(m^{p-1})=\frac{2^{-p}\pi^{\frac{p}{2}}}{\Gamma(1+\frac{p}{2})}m^p+O(m^{p-1}).$$
This concludes the proof of the first equality.

To see the second equality, note that
$$2^p\sum_{\substack{\mathbf{k}\in\mathbb{N}^p\\|\mathbf{k}|\leq m}}1\leq \sum_{\substack{\mathbf{k}\in\mathbb{Z}^p\\|\mathbf{k}|\leq m}}1\leq 2^p\sum_{\substack{\mathbf{k}\in\mathbb{Z}_+^p\\|\mathbf{k}|\leq m}}1.$$
The second equality follows now from \eqref{left right}.
\end{proof}

\begin{lemma}\label{nte2} Let $\alpha_p:\mathbb{Z}_+\to\mathbb{Z}_+^p$ be a bijection. If the mapping $m\to|\alpha_p(m)|$ increases, then
$$|\alpha_p(m)|^p=2^p\pi^{-p/2}\Gamma(1+\frac{p}{2})m+O(m^{\frac{p-1}{p}}).$$
Let $\alpha_p:\mathbb{Z}_+\to\mathbb{Z}^p$ be a bijection. If the mapping $m\to|\alpha_p(m)|$ increases, then
$$|\alpha_p(m)|^p=\pi^{-p/2}\Gamma(1+\frac{p}{2})m+O(m^{\frac{p-1}{p}}).$$
\end{lemma}
\begin{proof} It follows from Lemma \ref{nte1} that
$$m\leq\sum_{|\alpha_p(k)|\leq|\alpha_p(m)|}1=\sum_{\substack{\mathbf{k}\in\mathbb{Z}_+^p\\|\mathbf{k}|\leq |\alpha_p(m)|}}1\leq\sum_{\substack{\mathbf{k}\in\mathbb{Z}_+^p\\|\mathbf{k}|\leq\floor{|\alpha_p(m)|}+1}}1=$$
$$=\frac{2^{-p}\pi^{\frac{p}{2}}}{\Gamma(1+\frac{p}{2})}|\alpha_p(m)|^p+O(|\alpha_p(m)|^{p-1})$$
and
$$m\geq\sum_{|\alpha_p(k)|<|\alpha_p(m)|}1=\sum_{\substack{\mathbf{k}\in\mathbb{Z}_+^p\\|\mathbf{k}|<|\alpha_p(m)|}}1\leq\sum_{\substack{\mathbf{k}\in\mathbb{Z}_+^p\\|\mathbf{k}|\leq\floor{|\alpha_p(m)|}-1}}1=$$
$$=\frac{2^{-p}\pi^{\frac{p}{2}}}{\Gamma(1+\frac{p}{2})}|\alpha_p(m)|^p+O(|\alpha_p(m)|^{p-1}).$$
A combination of these estimates yields the first assertion and the proof of the second one is identical.
\end{proof}

\begin{lemma}\label{nte3} For every $p\in\mathbb{N},$ we have $(1-\Delta_p)^{-p/2}\in\mathcal{L}_{1,\infty}.$ For every normalised trace $\varphi$ on $\mathcal{L}_{1,\infty},$ we have
$$\Gamma(1+\frac{p}{2})\varphi((1-\Delta_p)^{-p/2})=\pi^{p/2}.$$
\end{lemma}
\begin{proof} It follows from Lemma \ref{nte2} that
$$\mu(m,(1-\Delta_p)^{-p/2})=\frac{\pi^{p/2}}{\Gamma(1+\frac{p}{2})}\frac1{m+1}+O((m+1)^{-1-\frac1p}).$$
Therefore,
$$\sum_{m=0}^n\mu(m,(1-\Delta_p)^{-p/2})=\frac{\pi^{p/2}}{\Gamma(1+\frac{p}{2})}\log(n+1)+O(1).$$
The assertion follows from Theorem \ref{spectral}.
\end{proof}

\section{An easy counter-example to formula~\ref{cfub formula}}\label{easy}

In Theorem \ref{hard example}, we required that both operators $T_1(1+D^2)^{-1/2}$ and $T_2(1+D^2)^{-1/2}$ are universally measurable. In this appendix, we show that a simpler counter-example with $T_2=1$ does exist if the requirement of universal measurability of $T_1(1+D_1)^{-p/2}$ is omitted. This gives a counter-example to the formula because one does not take into account the correction of the limiting process by powers (cf as in Lemma \ref{connes lemma}).

\begin{lemma} There exist a $0\leq T_1\in\mathcal{L}(l_2),$ a universally measurable operator $(1+D^2)^{-1/2}\in\mathcal{L}_{1,\infty}$ and a Dixmier trace\footnote{See \eqref{dixtrdef} for the definition of Dixmier trace.} ${\rm Tr}_{\omega}$ such that
$${\rm Tr}_{\omega}((T_1\otimes 1)(1+D^2\otimes 1+1\otimes D^2)^{-1})\neq\frac{\pi}{4}{\rm Tr}_{\omega}(T_1(1+D^2)^{-1/2}).$$
\end{lemma}
\begin{proof} Set $D={\rm diag}(\{k\}_{k\geq0})$ and $T_1={\rm diag}(\{x(k)\}_{k\geq0})$ with $x=\chi_{\cup_m[n_{2m},n_{2m+1})},$ where $\log(n_m)=o(\log(n_{m+1}))$ as $m\to\infty.$ Suppose that for every Dixmier trace ${\rm Tr}_{\omega}$ we have
\begin{equation}\label{loj}
{\rm Tr}_{\omega}\Big({\rm diag}\Big(\Big\{\frac{x(k)}{1+k^2+l^2}\Big\}_{k,l\geq0}\Big)\Big)=\frac{\pi}{4}{\rm Tr}_{\omega}\Big({\rm diag}\Big(\Big\{\frac{x(k)}{k+1}\Big\}_{k\geq0}\Big)\Big).
\end{equation}
In what follows, we omit ${\rm diag}$ to lighten the notations. Using definition \eqref{dixtrdef} of Dixmier traces, we can equivalently rewrite \eqref{loj} as
$$\lim_{n\to\omega}\frac1{\log(n+2)}\Big(\sum_{i=0}^n\mu\Big(i,\Big\{\frac{x(k)}{1+k^2+l^2}\Big\}_{k,l\geq0}\Big)-\frac{\pi}{4}\sum_{i=0}^n\mu\Big(i,\Big\{\frac{x(k)}{k+1}\Big\}_{k\geq0}\Big)\Big)=0$$
for every ultrafilter $\omega.$ Equivalently, we have
$$\sum_{i=0}^n\mu\Big(i,\Big\{\frac{x(k)}{1+k^2+l^2}\Big\}_{k,l\geq0}\Big)-\frac{\pi}{4}\sum_{i=0}^n\mu\Big(i,\Big\{\frac{x(k)}{k+1}\Big\}_{k\geq0}\Big)=o(\log(n)),\quad n\to\infty.$$
Lemma \ref{double kalton} states that
$$\sum_{i=0}^n\mu\Big(i,\Big\{\frac{x(k)}{1+k^2+l^2}\Big\}_{k,l\geq0}\Big)-\sum_{k,l=0}^{\floor{n^{1/2}}}\frac{x(k)}{1+k^2+l^2}=O(1),$$
while Lemma \ref{kalton} states that
$$\sum_{i=0}^n\mu\Big(i,\Big\{\frac{x(k)}{k+1}\Big\}_{k\geq0}\Big)-\sum_{k=0}^n\frac{x(k)}{k+1}=O(1).$$
We have
$$\sum_{k,l=0}^{\floor{n^{1/2}}}\frac{x(k)}{1+k^2+l^2}=\sum_{k=0}^{\floor{n^{1/2}}}x(k)\sum_{l=0}^{\floor{n^{1/2}}}\frac1{1+k^2+l^2}=$$
$$=\sum_{k=0}^{\floor{n^{1/2}}}x(k)\Big(O(\frac1{1+k^2})+\int_0^{n^{1/2}}\frac{dt}{1+t^2+k^2}\Big)=$$
$$=\sum_{k=0}^{\floor{n^{1/2}}}\frac{x(k)}{(k^2+1)^{1/2}}\tan^{-1}((\frac{n}{k^2+1})^{\frac12})+O(1).$$

Thus,
\begin{equation}\label{lojnoe uravnenie}
\frac{\pi}{4}\sum_{k=0}^n\frac{x(k)}{k+1}-\sum_{k=0}^{\floor{n^{1/2}}}\frac{x(k)}{k+1}\tan^{-1}((\frac{n}{k^2+1})^{\frac12})=o(\log(n)),\quad n\to\infty.
\end{equation}

We now show that \eqref{lojnoe uravnenie} actually fails. For $n=n_{2m}^2,$ we have
$$\sum_{k=0}^n\frac{x(k)}{k+1}\geq\sum_{k=n_{2m}}^{n_{2m}^2}\frac1{k+1}=\log(n_{2m})+O(1)=\frac12\log(n)+O(1)$$
and, taking into account that $x$ vanishes on the interval $[n_{2m-1},n_{2m}),$ we have
$$\sum_{k=0}^{\floor{n^{1/2}}}\frac{x(k)}{k+1}\tan^{-1}((\frac{n}{k^2+1})^{\frac12})\leq\frac{\pi}{2}\sum_{k=0}^{n_{2m-1}}\frac1{k+1}=\frac{\pi}{2}\log(n_{2m-1})+O(1)=o(\log(n)).$$
Hence, \eqref{lojnoe uravnenie} fails for such $x$ as $n=n_{2m}^2$ and $m\to\infty.$
\end{proof}


\begin{thebibliography}{99}
\bibitem{BF} Benameur M., Fack T. {\it Type II non-commutative geometry. I. Dixmier trace in von Neumann algebras.} Adv. Math. {\bf 199} (2006), no. 1, 29--87.
\bibitem{CRSZ} Carey A., Rennie A., Sukochev F., Zanin D. {\it Universal measurability and the Hochschild class of the Chern character.} J. Spectr. Theory {\bf 6} (2016), 1--41.
\bibitem{ch_pal} Chakraborty P., Pal A. {\it Equivariant spectral triples on the quantum SU(2) group.} K-Theory {\bf 28} (2003), no. 2, 107--126.
\bibitem{Connes} Connes A. {\it Noncommutative Geometry.} Academic Press, San Diego, 1994.
\bibitem{Connes_suq2} Connes A. {\it Cyclic cohomology, quantum group symmetries and the local index formula for SU$_q(2).$} J. Inst. Math. Jussieu {\bf 3} (2004), no. 1, 17--68.
\bibitem{Davidson} Davidson K. {\it $C^*-$algebras by example.} Fields Institute Monographs, {\bf 6}. American Mathematical Society, Providence, RI, 1996.
\bibitem{DFWW} Dykema K., Figiel T., Weiss G., Wodzicki M. {\it Commutator structure of operator ideals.} Adv. Math. {\bf 185} (2004), no. 1, 1--79.
\bibitem{DK} Dykema K., Kalton N. {\it Spectral characterization of sums of commutators. II.} J. Reine Angew. Math. {\bf 504} (1998), 127--137.
\bibitem{GMS} Gelfand I., Minlos R., Shapiro Z. {\it Representations of the rotation and Lorentz groups and their applications.} Oxford-London-New York-Paris: Pergamon Press. xviii, 366 pp. (1963); Moskva: Gosudarstv. Izdat. Fiz.-Mat. Lit., 368 pp. (1958).
\bibitem{Gilkey} Gilkey P. {\it Invariance theory, the heat equation, and the Atiyah-Singer index theorem.} Mathematics Lecture Series, {\bf 11}. Publish or Perish, Inc., Wilmington, DE, 1984.
\bibitem{greenbook} Gracia-Bondia J., Varilly J., Figueroa H. {\it Elements of noncommutative geometry.} Birkh\"auser Advanced Texts: Basler Lehrb\"ucher. Birkh\"auser Boston, Inc., Boston, MA, 2001.
\bibitem{Kcomm} Kalton N. {\it Spectral characterization of sums of commutators. I.} J. Reine Angew. Math. {\bf 504} (1998), 115--125.
\bibitem{KLPS} Kalton N., Lord S., Potapov D., Sukochev F. {\it Traces of compact operators and the noncommutative residue.} Adv. Math. {\bf 235} (2013), 1--55.
\bibitem{KSh} Klimyk A., Schm\"udgen K. {\it Quantum groups and their representations.} Texts and Monographs in Physics. Springer-Verlag, Berlin, 1997.
\bibitem{LPS} Lord S., Potapov D., Sukochev F. {\it Measures from Dixmier traces and zeta functions.} J. Funct. Anal. {\bf 259} (2010), no. 8, 1915--1949.
\bibitem{LSZ} Lord S., Sukochev F., Zanin D. {\it Singular Traces: Theory and Applications.} volume 46 of Studies in Mathematics. De Gruyter, 2013.
\bibitem{Rosenberg}  Rosenberg S. {\it The Laplacian on a Riemannian manifold. An introduction to analysis on manifolds.} London Mathematical Society Student Texts, {\bf 31}. Cambridge University Press, Cambridge, 1997.
\bibitem{SS} Sedaev A., Sukochev F. {\it Dixmier measurability in Marcinkiewicz spaces and applications.} J. Funct. Anal. {\bf 265} (2013), no. 12, 3053--3066.
\bibitem{SUZ1} Sukochev F., Usachev A., Zanin D. {\it Generalized limits with additional invariance properties and their applications to noncommutative geometry.} Adv. Math. {\bf 239} (2013), 164--189. 
\bibitem{SUZ2} Sukochev F., Usachev A., Zanin D. {\it Dixmier traces generated by exponentiation invariant generalised limits.} J. Noncommut. Geom. 8 (2014), no. 2, 321--336.
\bibitem{urojai} Sukochev F., Zanin D. {\it $\zeta-$function and heat kernel formulae.} J. Funct. Anal. {\bf 260} (2011), no. 8, 2451--2482.
\end{thebibliography}
\end{document}